\documentclass{amsart}
\usepackage{a4wide}

\usepackage{amsmath}
\usepackage{graphicx} 
\usepackage{amsfonts} 
\usepackage{amssymb}
\usepackage{color}
\usepackage{mathrsfs}
\usepackage{epsfig}
\usepackage{subfig}
\usepackage{url}

\graphicspath{{.}{figuresPL/}}
\usepackage[latin1]{inputenc}
\numberwithin{equation}{section}

\newtheorem{theorem}{Theorem}[section]
\newtheorem{lemma}[theorem]{Lemma}
\newtheorem{proposition}[theorem]{Proposition}
\newtheorem{corollary}[theorem]{Corollary}

\newtheorem{remark}[theorem]{Remark}

\newcommand{\ffi}{\varphi}
\newcommand{\ud}{\,\mathrm{d}}

\renewcommand\tilde{\widetilde}

\newcommand\e{\varepsilon}
\newcommand{\cu}{{c}}

\newcommand{\iniz}{{u_{0}}}
\newcommand{\ciniz}{{}}
\newcommand{\diniz}{{\delta_0}}
\newcommand{\newatop}{\genfrac{}{}{0pt}{1}}

\def\bdelta{\overline{\delta}}

\def\R{\mathbb{R}}

\newcommand\ep{\varepsilon}

\newcommand{\T}{\mathcal T}

\newcommand{\Str}{\Pi}
\newcommand{\yd}{y_d}

\newcommand{\etad}{{\eta_d}}
\newcommand{\xid}{\xi_d}
\newcommand{\xd}{x_d}

\newcommand\ges{\gtrsim}
\newcommand{\E}{\mathcal{E}}

\def\cC{\mathcal{C}}

\newcommand{\si}{l}

\newcommand{\finiz}{{f^0}}
\newcommand{\finizinf}{{\|\finiz\|_{L^\infty(S)}}}

\newcommand{\dfin}{{\bar{\delta}}}

\def\E{\mathcal{E}}

\def\Z{\mathbb{Z}}
\def\N{\mathbb{N}}
\def\K{\mathcal{K}}

\renewcommand{\phi}{\varphi}

\def\Div{\textup{div}\,}

\def\1{\mathbf{1}}
\def\loc{\mathrm{loc}}

\def\XXint#1#2#3{{\setbox0=\hbox{$#1{#2#3}{\int}$ }
\vcenter{\hbox{$#2#3$ }}\kern-.57\wd0}}

\renewcommand{\subset}{\subseteq}

\def\lt{\left}
\def\rt{\right}
\def\les{\lesssim}
\def\ges{\gtrsim}

\DeclareMathOperator*{\argmin}{\arg\!\min}

\def\weaklim{\rightharpoonup}

\def\E{\mathcal{E}}

\def\G{\mathcal{G}}

\title[Multi-dimensional Cahn-Hilliard relaxation rates] {A gradient flow approach to relaxation rates for the multi-dimensional Cahn-Hilliard equation}

\author[L. De Luca]
{L. De Luca}
\address[Lucia De Luca]{Via Bonomea 265, 34136 Trieste, Italy}
\email[L. De Luca]{ldeluca@sissa.it}
\author[M. Goldman]
{M. Goldman}
\address[Michael Goldman]{
Universit\'e Paris-Diderot, Sorbonne Paris-Cit\'e, Sorbonne Universit\'e,  CNRS,  Laboratoire Jacques-Louis Lions, LJLL, F-75013 Paris}
\email[M. Goldman]{goldman@math.univ-parisdiderot.fr}
\author[M. Strani]
{M. Strani}
\address[Marta Strani]{Dipartimento di Scienze di Base e Applicate per l'Ingegneria, Sapienza Universit\`a di Roma, Via Antonio Scarpa 10, Roma, Italy.
}
\email[M. Strani]{marta.strani@sbai.uniroma1.it}

\begin{document}

\begin{abstract}
The aim of this paper is to study relaxation rates for the Cahn-Hilliard equation in dimension larger than one. We follow the approach of Otto and Westdickenberg based on the gradient flow structure of the equation and establish differential and algebraic relationships between the energy,
the dissipation, and the squared $\dot H^{-1}$ distance to a kink. This leads to a scale  separation of the dynamics into two different stages: a first {\it fast} phase of the order $t^{-\frac 1 2}$ where one sees convergence to some kink, followed by a {\it slow} relaxation phase with rate $t^{-\frac 1 4}$ where convergence to the centered kink is observed.

%
%


\vskip5pt
\noindent
\textsc{Keywords.} gradient flow, relaxation to equilibrium, stability
\vskip5pt
\noindent
\textsc{AMS subject classifications.} 35K35, 35K55, 35B40

\end{abstract}
\maketitle
\section{Introduction}

In this paper we consider the Cahn-Hilliard equation on the strip $S:=Q\times \R\subset \R^d$ with $1\le d\le 5$
\begin{equation}\label{CH}
\left\{\begin{array}{ll}
 u_t-\Delta (G'(u)-\Delta u)=0 \quad & \   x \in S, \ t>0 \\
  u(0)=u_0,
\end{array}\right.
\end{equation}
where $Q$ is the $(d-1)$-dimensional torus defined by $Q:=\R^{d-1}/ \Z^{d-1}$ and $G$ is the standard double well potential $G(u):=\frac{(1-u^2)^2}{4}$. 

As first observed in \cite{Fi},  \eqref{CH}
can be seen as the gradient flow of the energy functional
\[
 E(u):=\int_{S} \frac{1}{2} |\nabla u|^2+G(u)\ud x
\]
with respect to the $\dot H^{-1}$ norm. Motivated for instance by the theory of phase transition, a very important class of stationary solutions to \eqref{CH} is given by the so-called {\it kink states}  or planar transition fronts defined as the one-dimensional functions
\begin{equation}\label{kinks}
\K:=\{v_c(\cdot)=v_{{0}}(\cdot-c)\,:\,c\in\R\}\,,
\end{equation}
where $v_0(z):=\tanh\lt(\frac {z}{\sqrt{ 2}}\rt)$ is the solution of
\begin{equation}\label{ELv}
 - v_{0zz}+G'(v_0)=0 \qquad v_0\to \pm 1 \textrm{ as } z\to \pm \infty, \quad v_0(0)=0\,.
\end{equation}
These are minimizers of the energy $E(u)$ under the condition that $u(x',z)\to \pm 1$ as $z\to \pm \infty$. Indeed, for any $w\in C^\infty(\R)$  with  $w(z)\to \pm 1$
as $z\to\pm\infty$, using the Modica-Mortola trick, we have 
$$
E_z(w):=\int_{\R}\frac 1 2|w_z|^2+G(w) \ud z\ge E_z(v)=\int_{-1}^1\sqrt{2\,G(s)}\ud s =:m_0\,,
$$
so that the functions in $\K$ are the only minimizers of $E_z$ among the functions with  $\pm 1$  boundary conditions at $\pm\infty$.
Moreover, for any function $u\in C^\infty(S)$ with $u(x',z)\to \pm 1$ as $z\to\pm\infty$ and for any $v\in\K$,  
\begin{multline*}
E(u)=\int_{Q}\int_{\R}  \frac{1}{2} |\nabla u|^2+G(u)\ud x'\ud z\ge \int_{Q}\int_{\R}\frac{1}{2} |\partial_z u|^2+G(u)\ud x'\ud z\ge \int_{Q} E_{z}(v)\ud x'=m_0\,,
\end{multline*}
so that the kink states are also minimizers of $E$\,.
The aim of this paper is to prove  asymptotic stability of these states together with optimal convergence rates. We  thus extend to higher dimensions previously known results in the $1d$ case (see \cite{OW} and references therein). \\
In our analysis we assume that the initial condition $u_0$ satisfies
\begin{equation}\label{masscons}
\int_S(u_0-v_0)\ud x=0\,.
\end{equation}
Our main result states that  solutions of \eqref{CH} with initial data  which are close enough to a kink state converge to $v_0$ at   two different time-scales: a {\it faster} time-scale on which the solution $u$ to \eqref{CH} converges to the $L^2$ closest kink $v_{c(t)}$ defined by 
\begin{equation}\label{minimaleint}
v_{c(t)}\in\argmin\{\|u(t)-v\|_{L^2}\,:\,v\in\K\}\,, 
\end{equation}
and a {\it slower} time-scale on which $v_{c(t)}$ converges to $v_0$.
Roughly speaking, in the first step the convergence is driven by the energy decay while in the latter the mass conservation plays a role through the assumption \eqref{masscons}\,.

Before precisely stating  the main theorem we introduce the three relevant  quantities, the relations among which will allow us to prove the desired decay rates.
These quantities are the energy gap, the dissipation and the squared distance defined respectively by
\begin{eqnarray}\label{gap}
\E(u)&:=&E(u)-m_0\\ \label{dissip}
 D(u)&:=&\int_{S} |\nabla (\Delta u-G'(u))|^2\ud x\\  \label{distance}
H(u)&:=&\inf_{\newatop{F\in L^2(S)}{\Div F= u-v_0}}\int_{S} |F|^2\ud x=\|u-v_0\|_{\dot{H}^{-1}}^2\,,
\end{eqnarray} 
where we have adopted 
 the usual convention that if the set $\{F\in L^2(S)\,:\,\Div F=u-v_0\}$ is empty, then $H(u)=+\infty$.
To shorten the notation, we set $\E_0:=\E(u_0)$, $D_0:=D(u_0)$ and $H_0:=H(u_0)$\,. Notice that if $H_0<\infty$ then necessarily \eqref{masscons} holds.

We are now in a position to state our main result.

 \begin{theorem}\label{teointro0}
Let $1\le d\le 5$. For every  $\overline{H}, \, \overline{\E}>0$, there exists $\diniz>0$ such that for every $u_0 \in L^\infty(S)$ with
$H_0 \le \overline{H}$, $\E_0 \leq \overline{\E}$ and $\|u_0 - {v}\|_{L^\infty} \leq \diniz$,
 there exists a unique solution $u=u(t,x)\in C^\infty((0,+\infty)\times S)$   of \eqref{CH}.\\
Moreover,  $u(t) \in L^\infty(S)\cap ({v_{c(t)}}+H^1(S))$ and
\begin{equation}\label{stimeteogenerico}
\begin{aligned}
\| u(t)-v_{c(t)} \|_{H^1} \lesssim  \mathcal{G}_0^{\frac{1}{2}} t^{-\frac{1}{2}}\,, \\
\|u(t)-v_{0}\|_{H^1}  \les \mathcal{G}_0^{\frac{1}{2}} t^{-\frac{1}{4}}\,,
\end{aligned}
\end{equation}
where $\mathcal{G}_0:= H_0+\E_0+\E^7_0$\,.
\end{theorem}
Let us observe that the discussion in \cite{OW} indicates that the  relaxation rates obtained in \eqref{stimeteogenerico} are optimal.
The proof closely follows the strategy laid in the case  $d=1$  by Otto and Westdickenberg in \cite{OW}.  It consists of  a non-linear energy-based method  which builds on the gradient flow structure of \eqref{CH} 
and combines algebraic and differential relationships between
$\E$, $H$ and $D$\, together with an ODE argument.  This scheme has been successfully implemented to obtain convergence rates to equilibrium in various related problems \cite{E,COW}\, and has recently proven pivotal to study coarsening rates for the Cahn-Hilliard equation (see \cite{SW}). However, since all these results deal with one dimensional
systems and since some of the central arguments used in \cite{OW} are one dimensional by nature, it was unclear if their strategy could be extended to higher dimensional problems. The main difference with the proof of \cite{OW} lies in the linearized energy gap and dissipation estimates where we replace the arguments of \cite{OW} by the use of the Lassoued-Mironescu trick \cite{LM}. 
Another main difference is that one of the most technical point in the analysis of \cite{OW} is to prove that the energy gap controls the $L^\infty$ distance of the solution to the kink. This allows the authors to make a Taylor expansion of the energy around the kink which is at the basis of most of the arguments. Unfortunately, in higher dimension this cannot be the case since having finite energy does not even guarantee to be bounded. 
For this reason, we need to assume that the initial data is close in $L^\infty$ to some kink and must prove that this property is preserved during the evolution. We thus obtain a perturbative result as opposed to the one in \cite{OW} where the only requirement besides $H_0<\infty$ is that $\E_0< 2 m_0$. Let us point out that since we assume 
somewhat {\it a priori} $L^\infty-$closeness to a kink, we can relax the requirement on the energy gap and merely assume that it is finite. The restriction on the dimension comes from the fact that as already underlined, 
we need to be able to prove that the $L^\infty$ distance to a kink remains small under the evolution which by the non-linear dissipation estimate  and the Gagliardo-Nirenberg inequality is possible if $d\le 5$ (see \eqref{nonlineardissip} and \eqref{gagliahd}).       
Even though the literature on the Cahn-Hilliard is very vast (see for instance \cite{Temam,ellsengmu,LWZ}), we were not able to find an existence result of classical solutions suiting our needs. Therefore, although it might be well-known to the experts, we decided to include a global existence result (see Theorem \ref{globalex}) for \eqref{CH}. The proof is inspired by \cite{LWZ} and is based on a Banach 
fixed point theorem in weighted H\"older spaces.\\

We now briefly recall the (few) results in the literature about the stability of planar wave solutions of related models  in higher dimensional domains. We refer again to \cite{OW} for a more detailed discussion. 
In the whole space case $\R^d$ with $d\ge 3$, Korvola, Kupiainen, and Taskinen \cite{KKT} established the asymptotics of the solution under the assumption that the initial datum is close to a kink state in a weighted $L^\infty$ norm.
In such a case, the translation of the front tends to zero as the time tends to infinity but the perturbation does not decay in the standard diffusive $t^{\frac 1 2}$ fashion  (as it is the case for us) but with a 
 $t^\frac{1}{3}$ scaling.
The method used by the authors is very different from our variational approach and is based on  a careful analysis of the semi-group generated by the linearization. 
A similar method has been used by Howard in \cite{H7b} to extend this stability result to $\R^2$.
The analysis by Howard makes use of pointwise estimates on the Green's function for the linearized operator in order to locate the shifts of the planar wave, through the local tracking method. 
This analysis has been extended by the same author to the non-linear case \cite{H7c,H7d} as well as to systems \cite{H7e}.\\
Let us mention that a similar (in spirit) result to ours has been obtained by Carlen and Orlandi in \cite{CO}. There, the authors study the stability of planar fronts for a non-local equation on the strip $S\subset\R^d$, where $d=2,3$\,. 
They prove that if the initial datum is $L^2$-close to a front and localized, then the solution  relaxes to another front in the $L^1$-norm. They also provide relaxation rates for the $L^2$-norm as well as rate of decrease for the free energy.
However, as already pointed out in \cite{OW} the method of \cite{CO} cannot give optimal rates.

\vskip0.2cm
The paper is organized as follows: 
\begin{itemize}
\item In Section \ref{section:prelim} we  prove some preliminary results which are used throughout the  paper.
\item Section \ref{section:endissest} is devoted to  the energy gap $\E$ and dissipation $D$ estimates.
\item In Section \ref{section:relaxest} we prove the main relaxation result Proposition \ref{prop:decay}. It is obtained by combining algebraic and differential relations between $\E$, $H$, $D$ and $c$ together with an ODE argument. 
\item Finally, in Section \ref{section:globalex}, we prove the first part of Theorem \ref{teointro0}, namely the global existence and uniqueness of smooth solutions to \eqref{CH}. 
We first obtain a local existence and uniqueness result (Theorem \ref{Tex0}) and then, using the estimates obtained in Section \ref{section:endissest}  and in Section \ref{section:relaxest}, we show that the solution  can be extended for any positive time. 
\end{itemize}

\vskip1mm
\noindent

{\bf Acknowledgments:}  
We warmly thank J.F. Babadjian, F. Cacciafesta and G. De Philippis  for very useful discussions related to the local existence result Theorem \ref{Tex0}. 
The hospitality of the Universit\'e Paris-Diderot and the Universit\`a di Roma  ``Sapienza'' where part of this research was done are gratefully acknowledged.
In the early stage of this work, LDL was funded by the DFG Collaborative Research Center CRC 109 ``Discretization in Geometry and Dynamics''. Part of this work was carried out while LDL was visiting Universit\`a di Roma ``Sapienza'' and  Universit\`a di Milano ``Bicocca'', thanks to the program ``Global Challenges for Women in Math Science''. MG was partially supported by the PGMO project COCA. MS  was supported in the early stage of this work  by the INdAM Fellowiship in Mathematics for Experienced Researchers cofounded by Marie Curie actions.

\vskip1cm
\textbf{Notation}
Let $Q:= \R^{d-1}/ \Z^{d-1}$ be the $(d-1)-$dimensional torus. We will always work in  the strip $S:=Q\times \R$. For any $x=(x_1,\ldots,x_{d-1},x_d)$, we use the notation  $x=(x',z)$ where $x':=(x_1,\ldots,x_{d-1})$ and $z:=x_d$. 
We denote by $\nabla'$ the gradient operator with respect to the variable $x'=(x_1,\ldots,x_{d-1})$ and set  
 $\nabla:=(\nabla',\partial_{z})$. Moreover, for the distributional derivatives of a function $w=w(t,x)$,
we will sometimes use the shorthand notation $w_{x_l}:=\partial_{x_l} w$  (for any $l=1,\ldots,d-1$) as well as  $w_z:=\partial_z w$ and $w_t:=\partial_t w$. Also, with a little abuse of notation and whenever the context is not ambiguous, $w(t)$ will denote the function $w(t,x)$ seen a function depending only on the space variable $x$. 
Finally, we write $A\les B$ if there exists a universal constant $C<\infty$ such that $A\le C\,B.$
We define $\ges$ analogously and say $A\sim B$ if $B\les A\les B$. We will also occasionally use the notation $\ll$ or $\gg$. For instance, $A(x)\ll B$ for $x\gg 1$ means that for every $\delta>0$ there exists $M<\infty$ such that $x\ge M$ implies $A(x)\le\delta\, B$. \\
As mentioned in the introduction, for any function $u\in L^2(S)$, we define the shifted kink $v_c(z):=v_0(\cdot-z)$ as an $L^2$ projection of $u$ onto the set $\K$ defined in \eqref{kinks}, i.e.,
\begin{equation}\label{minimale}
v_{c}\in\argmin\{\|u-v\|_{L^2}\,:\,v\in\K\}
\end{equation}
and we set $f_c:=u-v_c$\,.
As a consequence, $v_{\cu}$ satisfies the Euler-Lagrange equation
\begin{equation}\label{ELvc}
 \int_S f_\cu v_{\cu z}=\int_{S} (u-v_\cu)v_{\cu z} \ud x =0\,.
\end{equation}
Finally, for $u=u(t)$ we set $v_c(t):=v_\cu$\,.

\section{Preliminaries}\label{section:prelim}
In this section we gather a few simple observations and technical results which will be used later on in the paper. 

We first point out that, as in \cite[Remark 3]{OW}, if $\E(u)<\infty$ and $H(u)<\infty$, then $f_0\in H^1(S)$ and $u$ satisfies the right boundary conditions at infinity.
\begin{lemma}\label{lemH1}
 Let $u$ be such that $\E(u)<\infty$ and $H(u)<\infty$, then $f_0:=u-v_0\in H^{1}(S)$ and for a.e. $x'\in Q$, $u(x',z)\to \pm 1$ as $z\to \pm \infty$.
\end{lemma}
\begin{proof}
 If $u$ is such that  $\E(u)<\infty$ and $H(u)<\infty$, then
 \begin{equation*}
\begin{aligned}
\| \nabla f_0\|^2_{L^2}&=\int_{S} |\nabla f_0|^2\ud x\les \int_{S} |\nabla u|^2\ud x+\int_{S} |\nabla v_0|^2 \ud x\\
&\le E(u)+E(v_0)
\end{aligned}
\end{equation*}
and then, by interpolation,
\[
 \| f_0\|^2_{L^2}=\int_{S} f_0^2  \ud x\les \|f_0\|_{\dot H^{-1}} \|\nabla f_0\|_{L^2}\les H(u)(E(u)+E(v_0))^{\frac 1 2}\,,
\]
so that  $f_0$ is bounded in  $H^1$. As a consequence,
\[
\int_{Q} \sup_{z\ge M} f_0^2\ud x'\le \int_{Q}\ud x'\int_M^\infty (f_0^2+|\nabla f_0|^2)\ud z\to 0 \quad \textrm{ as } M\to \infty\,.
\]

Therefore $\sup_{|z|\ge M} f_0^2(\cdot,z)$ converges to zero in  $L^1(Q)$, and, since it is monotonically decreasing, it converges also a.e. to zero.
\end{proof}

Our second observation is that, as in the proof of \cite[Lemma 1.3]{OW}, closeness in $L^\infty$ to some kink implies closeness to $v_c$.

\begin{lemma}\label{lemma:fc}
 If $v_c$ is the minimizer of \eqref{minimale} for $u$, then for every $\bar c\in \R$,
 \begin{equation}\label{estimvinfty}
  \|f_c\|_{L^\infty}\les\|f_{\bar c}\|_{L^\infty}. 
 \end{equation}

\end{lemma}
\begin{proof}
 By triangle inequality, we have
 \[
  \|f_c\|_{L^\infty}\le \|f_{\bar c}\|_{L^\infty}+\|v_c-v_{\bar c}\|_{L^\infty}\les\|f_{\bar c}\|_{L^\infty}+\min(|c-\bar c|,1).
 \]
Now from \eqref{ELvc}, 
\[
 \min(|c-\bar c|,1)\les \lt|\int_{\R} (v_c-v_{\bar c}) v_{cz}  \ud z\rt|= \lt|\int_{S} (u-v_{\bar c}) v_{cz} \ud x\rt|\les \|f_{\bar c}\|_{L^\infty},
\]
from which \eqref{estimvinfty} follows.
\end{proof}
Combining Fourier arguments together with integration by parts, it is classically seen that the following holds. 
\begin{lemma}\label{Four}

Let $g\in L^2(S)$ with $\Delta g\in L^2(S)$; then $g\in H^2(S)$ with 
\begin{equation}\label{interpol}
\|\Delta g\|_{L^2(S)}=\|\nabla^2 g\|_{L^2(S)}  \quad \textrm{and} \quad \|\nabla g\|_{L^2(S)}^2\les \|g\|_{L^2(S)}  \, \|\nabla^2 g\|_{L^2(S)}.
\end{equation}

\end{lemma}
The last ingredient is a Hardy type inequality, completely analogous to \cite[Lemma 2.1]{OW}.
\begin{lemma}\label{2.1}
 For any $f_c \in H^1(S)$ such that 
 \begin{equation}\label{Elf1}
  \int_{S} f_c v_{cz} \ud x=0,
 \end{equation}
 it holds
 \[
  \int_{S}\frac{1}{(z-c)^2+1} f_c^2\ud x  \les  \int_{S} |\nabla f_c|^2\ud x.
 \]
\end{lemma}
\begin{proof}
The proof resembles the one in \cite[Lemma 2.1]{OW} with the slight difference that one needs also to control the transversal variable $x'$.\\
Up to a translation, {we may assume that $v_c=v_0=:v$}. Moreover, set $f_c=f_0=:f$, the claim becomes
 \begin{equation}\label{claim2.1}
  \int_{S}\frac{1}{z^2+1} f^2\ud x  \les  \int_{S} |\nabla f|^2\ud x.
 \end{equation}
 By the triangle inequality and the fact that $\int_\R \frac{1}{z^2+1}\ud z<\infty$, we have
 \begin{multline*}
   \int_{S}\frac{1}{z^2+1} f^2\ud z\les  \int_{S}\frac{1}{z^2+1} \lt(f-\int_{Q} f(y', z) \ud y'\rt)^2 \ud x \\
   +\int_{S}\frac{1}{z^2+1} \lt(\int_{Q} f(y', z)\ud y'-\int_{Q} f(y', 0)\ud y'\rt)^2 \ud x+ \lt(\int_{Q} f(y',0) \ud y'\rt)^2.
 \end{multline*}
The first term on the right-hand side can be estimated thanks to Poincar\'e inequality on $Q$ as
\begin{equation}\label{poincarehardy}
  \int_{S}\frac{1}{z^2+1} \lt(f-\int_{Q} f(y', z) \ud y'\rt)^2 \ud x\les  \int_{S}\frac{1}{z^2+1} |\nabla' f|^2 \ud x\les \int_{S} |\nabla' f|^2\ud x\ .
\end{equation}
Using Jensen's inequality and Hardy inequality in $\R$, we can estimate the second term as
\begin{equation}\label{quadr1}
\begin{aligned}
\int_{S}\frac{1}{z^2+1} \lt(\int_{Q} f(y', z)- f(y', 0)\ud y'\rt)^2 \ud x&\les \int_{Q} \int_{\R} \frac{1}{z^2+1} \left(f(x', z)-f(x',0)\right)^2\ud z\ud x' \\
&{=\int_Q\int_{\R}\frac{z^2}{z^2+1}\left(\frac 1 z\int_{0}^z \partial_z f(x',\zeta)\ud \zeta\right)^2\ud z\ud x'}\\
&\lesssim \int_{S} |\partial_{z} f(x',z)|^2 \ud x\,.
\end{aligned}
\end{equation}
We are left with estimating the last term. For this we observe that  \eqref{Elf1} and the fact that $\int_{S} v_z\ud x=2$ imply 
\begin{equation*}
\int_{Q} f(y',0) \ud y'= -\frac{1}{2}\int_{S} \lt(f-\int_{Q} f(y',0) \ud y'\rt)  v_{z} \ud x\,.
\end{equation*}
Using Cauchy-Schwarz inequality together with $|v_z|\les \frac{1}{1+z^2}$, we obtain

\[
\lt(\int_{Q} f(y',0) \ud y'\rt)^2\les \int_{S}\frac{1}{1+z^2} \lt(f-\int_{Q} f(y',0) \ud y'\rt)^2, 
\]
which using the triangle inequality \eqref{poincarehardy} and \eqref{quadr1} can be estimated by $\int_{S} |\nabla f|^2\ud x\,$. This concludes the proof of \eqref{claim2.1}.
\end{proof}

\section{The energy and dissipation estimates}\label{section:endissest}
In this section we prove the desired non-linear energy gap and dissipation estimates. As in \cite{OW} these are crucial ingredients in the proof of the relaxation rates in Proposition \ref{prop:decay}. We recall that $f_c=u-v_c$ where $v_c$ is a minimizer of the problem \eqref{minimale} and hence solves \eqref{ELvc}.
\begin{proposition}\label{nonlinear}
 There exists $\bdelta>0$ such that if $f_c\in H^3(S)$ and $\|f_c\|_{L^\infty}\le \bdelta$, then
 \begin{equation}\label{nonlineargap}
  \E\sim\int_{S} f_c^2+|\nabla f_c|^2\ud x,
 \end{equation}
and 
\begin{equation}\label{nonlineardissip}
  D\sim\int_{S} |\nabla f_c|^2+|\nabla^2 f_c|^2+|\nabla^3 f_c|^2\ud x.
 \end{equation}
\end{proposition}
As in \cite{OW}, these estimates are proven by obtaining first similar bounds for the corresponding linearized quantities and then using the smallness of $\|f_c\|_{L^\infty}$ to make a Taylor expansion. 
It is in the proofs of the linear estimates that we depart the most from the one dimensional arguments used in \cite{OW}. Let us point out that a direct consequence of Proposition \ref{nonlinear} and Gagliardo-Nirenberg inequalities is the following $L^\infty$ bound.
\begin{corollary}\label{cor:gaglia}
  There exists $\bdelta>0$ such that if $f_c\in H^3(S)$ and $\|f_c\|_{L^\infty}\le \bdelta$, then for $2\le d\le 5$
\begin{equation}\label{gagliahd}
  \|f_c\|_{L^\infty}\les \E^{\frac 1 2-\frac{d'}{12}} D^{\frac{d'}{12}},
\end{equation}
where $d':=\max(3,d)$.
\end{corollary}
\begin{proof}
 For $d=2$, \eqref{gagliahd} is a consequence of Proposition \ref{nonlinear} and of the Gagliardo-Nirenberg inequality
 \[
  \|f_c\|_{L^\infty}\les  \|f_c\|^{\frac 1 2}_{L^2}\|\nabla^2 f_c\|^{\frac 1 2}_{L^2} .
 \]
For $3\le d \le 5$,  \eqref{gagliahd} follows from  Proposition \ref{nonlinear} and another Gagliardo-Nirenberg inequality, namely
\[
 \|f_c\|_{L^\infty}\les \|f_c\|^{1-\frac d 6}_{L^2} \|\nabla^3 f_c\|^{\frac d 6}_{L^2} .
\]

\end{proof}

\subsection{Linear estimates}

Given $v_c \in \K$, we define the linearized energy gap of a function $f\in H^1(S)$ as
\[
E_\ell(f):=\int_{S} |\nabla f|^2 +G''(v_c) f^2\ud x
\] 
and the linearized dissipation as
\[
D_\ell(f):=\int_{S} |\nabla(-\Delta f+G'(v_c) f)|^2\ud x\,.
\]
In order to prove Proposition \ref{nonlinear}, we  start by showing the analogous estimates for the linearized quantities (see Lemma \ref{lingap} and Lemma \ref{lindis} below).
In the one-dimensional case,  such estimates rely on a rigidity argument developed in \cite[Lemma 3.4]{OW}. Since the proof 
of this result does not seem to extend easily  to the higher dimensional setting, we adopt a different approach (see Lemma \ref{lemrig} below) and use  the Lassoued-Mironescu trick \cite{LM} (see also \cite{GM,GMM} for applications of this idea in different contexts). 
Since $v_{cz}>0$, we can always write $f=v_{cz} g$ for some function $g$. 
Let us point out that   if  $f\in C_c^\infty(S)$ (respectively $f\in H^1(S)$), then $g:=\frac {f}{v_{cz}}\in C_c^\infty(S)$ (respectively  $g\in H^1_\loc(S)$).
\begin{lemma}\label{pos}
For any  $f\in H^1(S)$, it holds
\begin{equation}\label{claimlin}
E_\ell(f)=\int_{S} v_{cz}^2|\nabla g|^2\ud x\ge 0\,.
\end{equation}
\end{lemma}
\begin{proof}
We assume without loss of generality that $v_c=v_0=:v$\,.

\noindent
We first show that \eqref{claimlin} holds for $f\in C_c^\infty(S)$. By differentiating  \eqref{ELv}, we see that 
\begin{equation}\label{eguagliaa0}
-\Delta v_{z} +G''(v)v_{z}=0\,,
\end{equation}
which,  multiplying by a test function $\phi\in C_c^\infty(S)$ and integrating by parts, yields
\[
 \int_{S} \nabla v_{z}\cdot \nabla \phi+G''(v) v_{z} \phi \ud x=0\,.
\]
By choosing $\phi:=v_z g^2$ (which is in $C_c^\infty(S)$ since $f$ is), we obtain
\begin{equation}\label{LM1}
 \int_{S} g^2 |\nabla v_z|^2+2 v_z g\, \nabla v_z\cdot \nabla g+G''(v) v_z^2 g^2\ud x=0\,,
\end{equation}
whence \eqref{claimlin} follows, noticing that
\begin{equation*}
\begin{aligned}
E_\ell(f)&=\int_{S} v_z^2|\nabla g|^2 +g^2|\nabla v_z|^2 +2v_zg\,\nabla v_z\cdot \nabla g +G''(v) v_z^2 g^2\ud x\\
&=\int_{S}v_z^2|\nabla g|^2\ud x\,.
\end{aligned}
\end{equation*}

In order to show \eqref{claimlin} for  $f\in H^1(S)$, we adopt the following approximation argument. Let $f\in H^1(S)$ and let $\{f_n\}\subset C_c^\infty(S)$ be such that $\|f_n-f\|_{H^1}\to 0$ as $n\to\infty$. By \eqref{claimlin}, we have
\begin{equation}\label{converg}
\int_{S} v_z^2|\nabla g_n|^2\ud x=E_\ell(f_n)\to E_\ell(f)\qquad\textrm{ as }n\to\infty\,.
\end{equation}
Since $|v_{zz}|\les|v_z|$, we have 
\begin{align*}
 \lim_{n\to \infty} \int_{S} |v_z\nabla g_n-v_z \nabla g|^2 \ud x&\les \lim_{n\to \infty} \int_{S} |\nabla (v_z g_n)- \nabla (v_z g)|^2 \ud x+\int_{S} |v_{zz}(g_n-g)|^2 \ud x\\
 &\les \lim_{n\to \infty}\int_{S} |\nabla f_n-\nabla f|^2 \ud x+\int_{S} |f_n-f|^2 \ud x\\
 &=0\,.
\end{align*}
Therefore by \eqref{converg}
\[
 E_\ell(f)=\lim_{n\to \infty} E_\ell(f_n)=\lim_{n\to \infty} \int_{S} v_z^2|\nabla g_n|^2\ud x=\int_{S} v_z^2|\nabla g|^2\ud x\,.
\]
\end{proof}
We can now use \eqref{claimlin} to prove that up to a multiplicative factor, the only critical point of $E_\ell$ is $v_{cz}$.
\begin{lemma}\label{LemLM}
Let $v_c\in\K$. If $f\in H^1(S)$ is a solution of
\begin{equation}\label{lineq}
-\Delta f+G''(v_c)\,f=0\,,
\end{equation}
then $f=\alpha\,v_{cz}$ for some $\alpha\in\R$.
\end{lemma}
\begin{proof}
Before starting the proof, let us point out that by elliptic regularity, any weak solution of \eqref{lineq} is actually a smooth classical solution of this equation.\\
We assume without loss of generality that $v_c=v_0=:v$.
We preliminarily notice that, by \eqref{eguagliaa0} and integration by parts, $E_\ell(\alpha v_z)=0$ for any $\alpha\in\R$\,.
\noindent
We first show the claim assuming that $f$ minimizes $E_\ell$ in $H^1(S)$.
Indeed,  by Lemma \ref{pos},   we have 
$$
E_\ell(f)= \int_{S}v_z^2|\nabla g|^2\ud x\ge 0=E_\ell( v_z)\,,
$$
which implies that $E_\ell(f)=\int_{S}v_z^2|\nabla g|^2\ud x=0$ and, in turns, that  $f=\alpha\, v_z$ for some $\alpha\in\R$.

We thus only need to prove that if  $f$ is a solution of \eqref{lineq} in $H^1(S)$, then $f$ is a minimizer of $E_\ell$ in $H^1(S)$, i.e. $f$ satisfies
\begin{equation}\label{minim}
E_\ell(f+\ffi)\ge  E_{\ell}(f)\qquad\qquad\textrm{ for any }\ffi\in H^1(S)\,. 
\end{equation}
Consider first  $\ffi\in C_c^\infty(S)$.  Integrating by parts, we get
\begin{eqnarray*}
E_\ell(f+\ffi)&=&\int_{S}|\nabla (f+\ffi)|^2+G''(v)(f+\ffi)^2\ud x\\
&=&E_\ell(f)+2\int_{S}\nabla f\cdot\nabla \ffi+G''(v)f\ffi \ud x+E_\ell(\ffi)\\
&=&E_\ell(f)+2\int_{S}(-\Delta f+G''(v)f)\ffi\ud x+E_\ell(\ffi)=E_{\ell}(f)+E_{\ell}(\ffi)\ge E_{\ell}(f)\,,
\end{eqnarray*}
where in the last line we have used \eqref{lineq} and \eqref{claimlin}.
The general case $\ffi\in H^1(S)$ then follows by approximation.
\end{proof}

We can now prove the desired rigidity lemma.
\begin{lemma}\label{lemrig}
 Let $f$ be such that $\nabla f\in L^2(S)$ and $\nabla^2 f\in L^2(S)$  and let  $v_c\in\K$ be such that 
 \begin{equation}\label{Elf}
  \int_{S} f v_{cz}\ud x=0
 \end{equation}
  and 
 \begin{equation}\label{hyprig}
 -\Delta f +G''(v_c)f =\lambda
 \end{equation}
 for some $\lambda\in \R$. Then $f=0$.
\end{lemma}
\begin{proof}
As above, by elliptic regularity, every solution $f$ of \eqref{hyprig} is smooth. Again we assume $v_c=v_0=:v$\,.\\
\noindent
 Set $h(z):=\int_{Q} f(x',z)\ud x'$. Then, integrating \eqref{hyprig} and \eqref{Elf} with respect to $x'$, we get
 \[
  -h_{zz}+G''(v) h=\lambda \qquad \textrm{and} \qquad \int_{\R} hv_z\ud z=0\,,
 \]
where in the first equation we have used that for every $z\in \R$\,,
$$
{\int_{Q}-\Delta'f(x',z)\ud x'=0}\,.
$$
Therefore, \cite[Lemma 3.4]{OW} applies and $h=\alpha v_z$ for some $\alpha\in \R$\,. Since $ \int_{\R} hv_z\ud z=0$\,, this implies that $\alpha=0$, and, in turns, that $h=0$ and $\lambda=0$. 
By Poincar\'e inequality on $Q$ and the fact that $h=0$, for every $z\in \R$,
\[
 \int_Q f^2(x',z) \ud x'\les \int_Q |\nabla' f(x',z)|^2 \ud x',
\]
which after integration gives 
\[
 \int_S f^2 \ud x \les \int_S |\nabla f|^2\ud x
\]
so that $f\in H^2(S)$. Therefore, in view of \eqref{hyprig},  we may apply Lemma \ref{LemLM} and conclude that  $f=\alpha' v_z$ for some $\alpha'\in \R$. Since $h=0$, we must have $f=\alpha'v_z=0$.
\end{proof}

With this rigidity result in hand, we can prove a linear gap estimate, which is the analog of \cite[Lemma 3.1]{OW}. 

\begin{lemma}\label{lingap}
 For every $f\in H^1(S)$ and $v_c\in\K$ satisfying \eqref{Elf}, there holds
\begin{equation}\label{lingapf}
 E_{\ell}(f)\ges  \int_{S} f^2\ud x.
\end{equation}
\end{lemma}
\begin{proof}
 We argue by contradiction assuming that there exists a sequence $\{f_n\}\subset H^1(S)$ satisfying \eqref{Elf},
 \begin{equation}\label{hypabsurdlin}
 \int_{S} f_n^2\ud x=1
 \end{equation}
  and
  \begin{equation}\label{convEll}
\lim_{n\to\infty}E_\ell(f_n)= 0.
 \end{equation}
 Since $f_n$ is uniformly bounded in $H^1$, up to a subsequence, $f_n\weaklim f$ in $H^1$, for some $f\in H^1(S)$ satisfying \eqref{Elf}.  Moreover, the lower  semicontinuity of $E_\ell$, \eqref{convEll} and Lemma \ref{pos} imply that $E_\ell(f)=0$. 
 Hence, in view of  Lemma \ref{pos}, $f$ is a minimizer of $E_\ell$ in $H^1(S)$ and using  Lemma \ref{LemLM} together with  \eqref{Elf} we obtain $f=0$. 

\noindent
We finally prove that $\|f_n\|_{L^2}\to 0$, which provides a contradiction to \eqref{hypabsurdlin}. 
For this purpose, we first show that 
\begin{equation}\label{toprovelin}
 \lim_{n\to \infty}\int_{S} |G''(1)-G''(v_c)| f_n^2\ud x=0\,.
\end{equation}
Let $Z>0$ to be fixed. Since $Q\times(-Z,Z)$ is bounded, $f_n\to f\equiv 0$ strongly in $L^2(Q\times(-Z,Z))$. Therefore on the one hand
\[
\lim_{n\to \infty} \int_{-Z}^{Z}\int_{Q}|G''(1)-G''(v_c)|f_n^2\ud x'\ud z=0\,.
\]
On the other hand, since $G''(v_c(z))\to G''(1)$ as $z\to \pm \infty$, from \eqref{hypabsurdlin}, we get
\[
 \int_{|z|\ge Z} \int_{Q} |G''(1)-G''(v_c)|f_n^2\ud x'\ud z\le \sup_{|z|\ge Z} |G''(1)-G''(v_c)|
\]
which tends to zero as $Z\to +\infty$ uniformly in $n$. This proves \eqref{toprovelin}.  Finally, by \eqref{toprovelin} and  \eqref{hypabsurdlin}, we obtain
\[
0= \lim_{n\to +\infty}\int_{S} (G''(1)-G''(v_c)) f_n^2\ud x+E_\ell(f_n)=\lim_{n\to +\infty}\int_{S} |\nabla f_n|^2 +G''(1) f_n^2\ud x\,, 
\]
and, since $G''(1)>0$, this implies
\[  \int_{S} f_n^2\ud x \to 0,\]
thus concluding the proof of Lemma \ref{lingap}.
\end{proof}
We now turn to the dissipation estimate (see \cite[Lemma 3.2]{OW}).
\begin{lemma}\label{lindis}
 For any $f\in H^3(S)$ satisfying \eqref{Elf}, it holds
 \[
  D_\ell(f)\ges  \int_{S} |\nabla f|^2\ud x.
 \]

\end{lemma}

\begin{proof}
Once again, we may assume without loss of generality that $v_c=v_0=:v$.

We argue by contradiction and assume that there exists a sequence of functions $f_n \in H^3(S)$ such that \eqref{Elf},
\begin{equation}\label{2hyp}
 \int_{S} |\nabla f_n|^2\ud x=1
\end{equation}
and
\begin{equation}\label{3hyp}
D_\ell(f_n) \to 0 
\end{equation}
hold. First of all, we claim that \eqref{2hyp} improves to
\begin{equation}\label{somma1}
 \int_{S} \frac{1}{z^2+1}f_n^2 + |\nabla f_n|^2+ |\nabla^2 f_n|^2 + |\nabla^3 f_n|^2 \ud x\lesssim 1\,.
\end{equation}
Indeed, by Lemma \ref{2.1}
\begin{equation}\label{conto1} 
\int_{S} \frac{1}{z^2+1}f_n^2\ud x \lesssim \int_{ S} |\nabla f_n|^2\ud x\,.
\end{equation}
Furthermore, by \eqref{3hyp}, we have
\begin{equation*}
\begin{aligned}
 \int_{S} |\nabla  \Delta f_n|^2 \ud x&\lesssim  \int_{S} \left|\nabla(\Delta f_n- G''(v)f_n)\right|^2\ud x+ \int_{S} | \nabla (G''(v)f_n )|^2 \ud x\\
& = D_\ell(f_n)+ \int_{S}  | \nabla (G''(v)f_n )|^2 \ud x\\
& \lesssim 1 +\int_{S}  | \nabla (G''(v)f_n )|^2 \ud x\,.
 \end{aligned}
 \end{equation*}
Using that $|G^{(3)}(v)v_z|\lesssim \frac{1}{1+z^2}$ and applying Lemma \ref{2.1}, we obtain
\begin{equation*}
\begin{aligned}
\int_{S}  | \nabla (G''(v)f_n )|^2 \ud x&\lesssim \int_{S} \left(  G^{(3)}(v) v_z f_n\right)^2\ud x + \int_{S} \left|  G''(v)\nabla f_n\right|^2  \ud x \\
&\lesssim \int_{S} |\nabla f_n|^2\ud x\,.
\end{aligned}
\end{equation*}
This fact,  together with Lemma \ref{Four}, implies
\begin{equation}\label{estimdeltagrad}
\int_{S} |\nabla^3 f_n|^2\ud x = \int_{S} |\nabla  \Delta f_n|^2 \ud x\les 1+ \int_{S} |\nabla f_n|^2\ud x\, .
\end{equation}
Finally, by applying Lemma \ref{Four} to $g=\nabla f_n$, we get
\begin{equation}\label{conto3}
\int_{S} |\nabla^2 f_n|^2 \ud x\les \int_{S} |\nabla f_n|^2\ud x\,.
\end{equation}
By summing \eqref{conto1}, \eqref{estimdeltagrad} and \eqref{conto3}, and recalling \eqref{2hyp}, formula \eqref{somma1} follows.

In light of \eqref{somma1} and \eqref{3hyp}, $f_n$ weakly converges, up to subsequences, to a function $f \in H^3_\loc(S)$ satisfying \eqref{Elf},
\[
\int_{S}  |\nabla f|^2+ |\nabla^2 f|^2 + |\nabla^3 f|^2 \ud x\lesssim 1, \]
and $D_\ell(f)=0$.
It follows that
\begin{equation*}
\nabla(-\Delta f+ G''(v)f)=0,
\end{equation*}
and hence, by Lemma \ref{lemrig}, we get $f=0$. Arguing as in the proof of Lemma \ref{lingap}, it can be proven that   $\|\nabla f_n \|_{L^2} \to 0$, thus obtaining a contradiction.

\end{proof}

\subsection{Proof of the non-linear estimates}
We can now prove Proposition \ref{nonlinear}. We start with the non-linear energy gap estimate \eqref{nonlineargap}.
We recall that $f_c=u-v_c$ with $v_c$ solving \eqref{ELvc} and
\begin{equation}\label{formula:e}
 \E(u)= \int_{S} \frac{1}{2} |\nabla u|^2 +G(u) \ud x-\int_{S} \frac{1}{2} |\nabla v_c|^2 +G(v_c) \ud x\,.
\end{equation}
\begin{lemma}\label{proof1.18}
There exists $\bdelta>0$ such that if $\|f_c\|_{L^\infty}\le \bdelta$, then 
\begin{equation}\label{1.18}
\E\sim\int_{S} |f_c|^2+|\nabla f_c|^2\ud x.
\end{equation}
\end{lemma}
\begin{proof}
We may assume that $\|f_c\|_{L^\infty}\les 1$. We start by proving that
\begin{equation}\label{1.18ineq}
\E\les\int_{S} |f_c|^2+|\nabla f_c|^2\ud x.
\end{equation}
Indeed, by \eqref{formula:e} and using that
$$
0=\int_{S}(-v_{czz}+G'(v_c))f_c\ud x=\int_{S}(-\Delta v_{c}+G'(v_c))f_c\ud x=\int_{S}\nabla v_c\cdot\nabla f_c+G'(v_c)f_c\ud x\,,
$$
we have
\begin{equation}\label{primopasso}
\E=\int_{S}\frac 1 2|\nabla f_c|^2+G(u)-G(v_c)-G'(v_c)f_c\ud x\,.
\end{equation}

Now since  $\|f_c\|_{L^\infty}\les 1$, by Taylor expansion we have
\[
|G(u)-G(v_c)-G'(v_c)f_c| \lesssim |f_c|^2,
\]
so that
\begin{equation}\label{stima4}
\int_{S}|G(u)-G(v_c)-G'(v_c)f_c|\ud x\les \int_{S}|f_c|^2\ud x\,.
\end{equation}
Combining this with \eqref{primopasso} yields \eqref{1.18ineq}.

In order to prove that 
\begin{equation}\label{1.18ineqopp}
\E\ges\int_{S} |f_c|^2+|\nabla f_c|^2\ud x\, ,
\end{equation}
we first show that 
\begin{equation}\label{secondopasso}
\int_{S} |f_c|^2 \ud x \lesssim \E\, .
\end{equation}
Using again the hypothesis $\|f_c\|_{L^\infty}\les 1$ and a Taylor expansion, we get 
\begin{align*}
 \E&= \frac{1}{2}\int_{S}  |\nabla f_c|^2+ G''(v_c) f_c^2 \ud x +\int_{S} G(u)-G(v_c)-G'(v_c)f_c- \frac{1}{2}  G''(v_c) f_c^2 \ud x\\
 &\ges E_\ell(f_c) -\|f_c\|_{L^\infty} \int_{S} f_c^2 \ud x\, .
\end{align*}
Using \eqref{lingapf} we obtain that \eqref{secondopasso} holds if $\|f_c\|_{L^\infty}$ is small enough.
 
Finally, by \eqref{primopasso},  \eqref{stima4}, \and \eqref{secondopasso}, we immediately have
\begin{equation*}
\begin{aligned}
\int_{S}|\nabla f_c|^2\ud x &\les \E +\int_{S}| G(u)-G(v_c)-G'(v_c)f_c|\ud x \\
&  \les \E + \int_{S} |f_c|^2\ud x\\
&\les \E,
\end{aligned}
\end{equation*}
which proves \eqref{1.18ineqopp} and concludes the proof of the lemma.
\end{proof}
We end this section by proving the non-linear dissipation estimate \eqref{nonlineardissip}. Let us recall that 
\[
 D=\int_{S} |\nabla (\Delta u-G'(u))|^2 \ud x\,.
\]

\begin{lemma}\label{proof1.19}
There exists $\bdelta>0$ such that if $f_c\in H^3(S)$ and $\|f_c\|_{L^\infty}\le \bdelta$, then
\begin{equation}\label{1.19}
D \sim \int_{S} |\nabla f_c|^2+|\nabla^2 f_c|^2+|\nabla^3 f_c|^2 \ud x\,. 
\end{equation}
\end{lemma}
\begin{proof}
We may assume that $\|f_c\|_{L^\infty}\les 1$.
We start by proving that 
 \begin{equation}\label{1.19ineq}
 D\lesssim \int_{S} |\nabla f_c|^2+|\nabla^2 f_c|^2+|\nabla^3 f_c|^2 \ud x\,.
 \end{equation}
 Using the identity
 \begin{equation}\label{iddissipation}
 \nabla(-\Delta u +G'(u))=-\nabla \Delta f_c + \nabla (G'(u)-G'(v_c))\,,
 \end{equation}
by Lemma \ref{Four} and by integration by parts, we have
 \begin{equation}\label{sviluppoD}
 \begin{aligned}
 D&= \int_{S} |\nabla(\Delta u -G'(u))|^2 \ud x\\
& \leq \int_{S} |\nabla \Delta f_c|^2 + |\nabla(G'(u)-G'(v_c))|^2 \ud x\\
& = \int_{S}|\nabla^3 f_c|^2+  |\nabla(G'(u)-G'(v_c))|^2 \ud x\,.
 \end{aligned}
 \end{equation}
By Taylor expansion and by the assumption $\|f_c\|_{L^\infty}\les 1$, we have
 \begin{equation}\label{tay}
 \begin{aligned}
|\nabla(G'(u)-G'(v_c))|^2&=|(G''(u)-G''(v_c))\nabla v_c+G''(u)\nabla f_c|^2\\
&\les |v_{cz}|^2|f_c|^2+|\nabla f_c|^2.
\end{aligned}
 \end{equation}
Using  Lemma \ref{2.1} and the exponential decay of $v_{cz}$, we get 
\begin{equation}\label{stimaG'}
\begin{aligned}
\int_{S}|\nabla(G'(u)-G'(v_c))|^2 \ud x
& \les \int_{S} | v_{cz}|^2 |f_c|^2 + |\nabla f_c|^2\ud x \\
& \les  \int_{S}   |\nabla f_c|^2\ud x\,,
\end{aligned}
\end{equation}
which combined with \eqref{sviluppoD} yields \eqref{1.19ineq}.
%
%

We now prove that
\begin{equation}\label{opposite}
 D \gtrsim \int_{S} |\nabla f_c|^2+|\nabla^2 f_c|^2+|\nabla^3 f_c|^2 \ud x\,.
 \end{equation}
 First of all, we show that 
 \begin{equation}\label{claim}
 D\gtrsim \int_{S} |\nabla f_c|^2\ud x\,.
 \end{equation}
Indeed, by  \eqref{iddissipation} and the triangle inequality, we have
\begin{equation}\label{differenza}
\begin{aligned}
D &= \int_{S} |\nabla(\Delta f_c - G'(u)-G'(v_c)) |^2\ud x  \\
&\ges \int_{S} |\nabla(\Delta f_c-G''(v_c)f_c)|^2\ud x-\int_{S}|\nabla(G'(u)-G'(v_c)-G''(v_c)f_c)|^2\ud x\\
&\ges \int_{S}|\nabla f_c|^2\ud x -\int_{S}|\nabla(G'(u)-G'(v_c)-G''(v_c)f_c)|^2\ud x\,,
\end{aligned}
\end{equation}
where the last inequality is an immediate consequence of Lemma \ref{lindis}.
Moreover,  arguing as in \eqref{tay}, we get
\begin{equation*}
\begin{aligned}
|\nabla(G'(u)-G'(v_c)-G''(v_c)f_c)|^2&=|(G''(u)-G''(v_c)-G^{(3)}(v_c)f_c)\nabla v_c+(G''(u)-G''(v_c))\nabla f_c |^2 \\
&\les |f_c|^4 |v_{cz}|^2 +|f_c|^2 |\nabla f_c|^2\\ 
&\les \|f_c\|_{L^\infty}^2 (|f_c|^2|v_{cz}|^2+|\nabla f_c|^2)\,,
\end{aligned}
\end{equation*}
from which, using again Lemma \ref{2.1} and the exponential decay of $v_{cz}$, we obtain
\begin{equation}\label{mancante}
\begin{aligned}
\int_{S}|\nabla(G'(u)-G'(v_c)-G''(v_c)f_c)|^2 \ud x
& \les  \|f_c\|_{L^\infty}^2 \int_{S} v_{cz}^2 f_c^2+ |\nabla f_c|^2\ud x \\
& \les  \|f_c\|_{L^\infty}^2 \int_{S} |\nabla f_c|^2\ud x\,.
\end{aligned}
\end{equation}
Assuming that $ \|f_c\|_{L^\infty}\le\bdelta$ for some $\bdelta$ sufficiently small, \eqref{opposite} follows by \eqref{differenza} and \eqref{mancante}.
\noindent
Finally we show that \eqref{claim} improves to \eqref{opposite}.
Indeed, by the  definition of $D$, \eqref{iddissipation},\eqref{stimaG'}, and \eqref{claim}, we have
\begin{equation*}
\begin{aligned}
\int_{S} |\nabla \Delta f_c|^2\ud x& \les   \int_{S}|\nabla(-\Delta u +G'(u))|^2\ud x +\int_{S} |\nabla (G'(u)-G'(v_c))|^2\ud x\\
 &\les D+\int_{S}|\nabla f_c|^2\ud x\les D\,.
\end{aligned}
\end{equation*}
Using \eqref{interpol}, this proves  \eqref{opposite}.
\end{proof}
\section{The main relaxation estimate}\label{section:relaxest}
The aim of this section is to prove the following result, which gives the relaxation rates to equilibrium.
\begin{proposition}\label{prop:decay}
 Let $\bdelta$ be given by Proposition \ref{nonlinear} and let $T>0$ be fixed. Then, every smooth solution $u$ of \eqref{CH} in $(0,T)\times S$ such that  $H(u(t))<\infty$,
 $f_{c}(t)\in H^j(S)$ for every $j\ge 0$ and 
 $\|f_{c}(t)\|_{L^\infty}\le \bdelta$ for every $t\in (0,T)$, satisfies
 \begin{eqnarray}\label{stimasuE}
&\E(u(t)) &\le \E_0\,, \\ \label{1intro}
&\E(u(t)) &\les \mathcal{G}_0t^{-1}\,, \\ 
&c^2(t) &\les \mathcal{G}_0^{\frac{1}{2}}\E^{\frac{1}{2}}_0\,, \\ 
\label{2intro}
&c^2(t) &\les \mathcal{G}_0t^{-\frac{1}{2}}\,, \\ \label{stimasuH}
&H(u(t))& \les\mathcal{G}_0\,,\\ \label{stimadissipazione}
&D(u(t)) & \les (\mathcal{G}_0+\mathcal{G}_0^{{2}}+ \E_0\, \mathcal{G}_0^{\frac{6}{6-d'}} t^{-\frac{2d'-6}{6-d'}})t^{-2}\,,
\end{eqnarray}
where $\mathcal{G}_0:= H_0+\E_0+\E_0^7$ and $d':=\max(3,d)$.
\end{proposition}

\begin{remark}\label{rem:new}
\rm{We note that Proposition \ref{prop:decay} and  Lemma \ref{lemmaintro} below  imply  \eqref{stimeteogenerico}. Indeed, \eqref{nonlineargap} and \eqref{1intro}, yield
\begin{equation}\label{un}
\| u( t)-v_{c{(t)}} \|_{H^1} =\|f_{c}(t)\|_{H^1}\sim \E^{\frac 1 2}(u(t))\lesssim  \mathcal{G}_0^{\frac{1}{2}} t^{-\frac{1}{2}}\,,
\end{equation}
which is exactly the first estimate in \eqref{stimeteogenerico}.

Moreover, by \eqref{un} and \eqref{2intro}, we get
\begin{equation*}
\begin{aligned}
\|u(t)-v_{0}\|_{H^1} & \leq \| u(t)-v_{c(t)} \|_{H^1} + \|v_{c{(t)}}-v_0 \|_{H^1} \\
& \les\mathcal{G}_0^{\frac{1}{2}} t^{-\frac{1}{2}}+ |c{(t)}| \\
& \les\mathcal{G}_0^{\frac{1}{2}} t^{-\frac{1}{4}}\,,
\end{aligned}
\end{equation*}
which coincides with the second estimate in \eqref{stimeteogenerico}.
}
\end{remark}

As in \cite[Theorem 1.2]{OW}, the proof of Proposition \ref{prop:decay} is based on the combination of  algebraic and differential relations between $\E$, $D$, $H$ and $c$ together with an ODE argument.\\

First, arguing almost verbatim as in \cite{OW}, and using Proposition \ref{nonlinear} and Lemma \ref{2.1}  the following result  can be proven (we refer the reader to \cite[Lemma 1.3]{OW} for the proof).
\begin{lemma}\label{lemmaintro}
Let  $\bar\delta$ be given by Proposition \ref{nonlinear} and let $u\in v_c+H^3(S)$ be such that $H(u)<\infty$ and $\|f_c\|_{L^\infty}\le \bar\delta$,  then
\begin{eqnarray}\label{stimeutiliintro0}
c^2 &\les &(H(u) \E(u))^{1/2}+(|c|+1) \E(u), \\ \label{stimeutiliintro3}
\E(u) &\les &(H(u)D(u))^{1/2} + (|c|+1)^2 D(u).
\end{eqnarray}
\end{lemma}
We remark that the assumptions of Lemma \ref{lemmaintro}  imply that $\E(u)<\infty$.\\


We now turn to the differential relations. The proof follows  the lines of the proof of \cite[Lemma 1.4]{OW} with some additional difficulties coming from the transversal directions.
\begin{lemma}\label{lemmaintrobis}
Let $\bar\delta$ be given by Proposition \ref{nonlinear} and let $T>0$ be fixed. Let  $u=u(t,x)$ be a smooth  solution of \eqref{CH} in $S\times (0,T)$ with $H(u(t)),\E(u(t))<\infty$ and $u(t)\in v_{c(t)}+H^3(S)$ for every $t\in [0,T)$\,. If   $\|f_{c}(t)\|_{L^\infty}\le \bdelta$ for every $t\in [0,T]$, then $u$ satisfies 
\begin{eqnarray}\label{stimeutiliintro4a}
\frac{\ud\E}{\ud t}&=& -D,\\  \label{stimeutiliintro4b}
\frac{\ud H}{\ud t} &\les&  ((|c|+1)c^2 \, D)^{\frac{1}{2}}+\E^{\frac{3}{2}-\frac{d'}{12}} \, D^{\frac{d'}{12}}, \\  \label{stimeutiliintro4c}
\frac{\ud D}{\ud t} &\les&  D^{\frac{3}{2}} +\E^{1-\frac{d'}{6}} D^{1+\frac{d'}{6}}, 
\end{eqnarray}
where $d':=\max(3,d)$.
\end{lemma}
\begin{remark}
\rm{
 Notice that if $\E\les 1$ and $d\le 3$, \eqref{stimeutiliintro4b} and \eqref{stimeutiliintro4c} are exactly the estimates of \cite[Lemma 1.4]{OW}. Moreover, if we still assume that $\E\les 1$ but that $d=4,5$, we expect that  $D\ll 1$ after a time of order one, and hence   \eqref{stimeutiliintro4b} and \eqref{stimeutiliintro4c} give very similar bound to the ones in \cite[Lemma 1.4]{OW}.
}
\end{remark}

\begin{proof}[Proof of Lemma \ref{lemmaintrobis}]

 Identity \eqref{stimeutiliintro4a} is a direct consequence of the definitions of $\E$ and $D$ and of  the fact that $\nabla u\in H^2(S)$.\\
 
 {\it Step 1:} Proof of \eqref{stimeutiliintro4b}.\\
 By  definition of $f_0$ and $f_c$, we have
\begin{equation}\label{eqf0}
 \partial_t f_{0}-\Delta \lt(G'(v_c+f_c)-G'(v_c)-\Delta f_c\rt)=0\,.
\end{equation}
We first show that \eqref{eqf0} implies
\begin{equation}\label{eqF}
 \partial_t F_{0}-\nabla\lt(G'(v_c+f_c)-G'(v_c)-\Delta f_c\rt)=0\,,
\end{equation}
where $F_0\in L^2(S)$ is a solution of the minimum problem in \eqref{distance}, i.e., 
\begin{equation}\label{Fzero2}
\Div F_0=f_0\qquad\textrm{and}\qquad H=\int_{S}F_0^2\ud x\,.
\end{equation}
Notice that \eqref{eqF} in particular shows that under the assumptions of the lemma, $H$ is differentiable in time.\\
Let $F\in L^2(S)$ with $\Div F=f_0$. Setting $\xi=(\xi',\xi_z)\in\Z^{d-1}\times\R$, and writing $F$ and $f_0$ in Fourier space, we get 
\[
F(x)=\sum_{\xi' \in \Z^{d-1}} \int_{\R} \hat{F}(\xi',\xi_z) e^{2 i \pi\xi\cdot x} \ud \xi_z \quad \mbox{and }f_0(x)=\sum_{\xi'\in\Z^{d-1}}\int_{\R} \hat{f}_0(\xi',\xi_z) e^{2 i \pi\xi\cdot x} \ud \xi_z\,;
\]
therefore,  the constraint $\Div F=f_0$ can be rewritten as $ 2i \pi \hat{F}\cdot \xi=\hat{f}_0$, which,  together with Plancherel Theorem, implies
\begin{equation}\label{h-1norm}
 \|f_0\|_{H^{-1}}^2=\min_{\newatop{\hat F\in L^2}{ 2i \pi \hat{F}\cdot \xi=\hat{f}_0}} \sum_{\xi'\in \Z^{d-1}}\int_{\R} |\hat{F}(\xi)|^2\ud\xi_z\,.
\end{equation}
Now,  if $\hat F_0$ is a minimizer of \eqref{h-1norm},  by Pythagoras  for every $\xi\neq0$, $\hat{F}_0(\xi)\in \R \xi$ and hence
\begin{equation}\label{minimofourier}
\hat{F}_0(\xi)=\frac{\hat{f}_0(\xi)}{2 i \pi} \frac{\xi}{|\xi|^2}\,.
\end{equation}
Set $g:= G'(v_c+f_c)-G'(v_c)-\Delta f_c$,  by \eqref{eqf0} we have
\[
\partial_t\hat{F}_0(\xi)= \frac{\partial_t \hat{f}_{0}(\xi)}{2 i\pi} \frac{\xi}{|\xi|^2}= 2 i \pi \hat{g}(\xi) \xi\,,
\]
whence \eqref{eqF} follows by taking  the Fourier inverse. 
 By \eqref{eqF} and integrating by parts, we have 
\begin{equation*}
\begin{aligned}
\frac{1}{2}\frac{\ud}{\ud t} \int_{S} F_0^2 \ud x &= \int_{S} F_0\, \partial_t F_{0} \ud x \\
&=\int_{S} F_0\cdot  \nabla\lt(G'(v_c+f_c)-G'(v_c)-\Delta f_c\rt) \ud x \\
&\stackrel{\Div F= f_0}{=}-\int_{S} f_0   \lt(G'(v_c+f_c)-G'(v_c)-\Delta f_c\rt) \ud x \\
&=-\int_{S} (f_c+v_c-v_0)\lt(G'(v_c+f_c)-G'(v_c)-\Delta f_c\rt) \ud x\\
&\le-\int_{S}|\nabla f_c|^2+G''(v_c)f_c^2 \ud x\\
&\quad -\int_{S}\lt(G'(v_c+f_c)-G'(v_c)-G''(v_c) f_c\rt) f_c \ud x \\
& \quad +\int_{S}(v_0-v_c)\lt(G'(v_c+f_c)-G'(v_c)\rt) \ud x + \int_{S} |(v_0-v_c)_z|\,|\nabla f_c| \ud x\\
&\le -\int_{S}\lt(G'(v_c+f_c)-G'(v_c)-G''(v_c) f_c\rt) f_c \ud x \\
& \quad +\int_{S}(v_0-v_c)\lt(G'(v_c+f_c)-G'(v_c)\rt) \ud x + \int_{S} |(v_0-v_c)_z| |\nabla f_c| \ud x\\
&=: A+B+C.
\end{aligned}
\end{equation*}
where the last inequality follows directly by the positivity of the linearized energy gap proved in Lemma \ref{lingap}.

We first estimate $A$. Using Taylor expansion, boundedness of $\|f_c\|_{L^\infty}$, and  \eqref{gagliahd},  
we have
\begin{equation}\label{A}
A\les\int_{S}|f_c|^3\ud x\le\|f_c\|_{L^\infty}\int_{S}|f_c|^2\ud x\les\E^{\frac 3 2-\frac{d'}{12}}D^{\frac{d'}{12}}.
\end{equation}

The estimates of $B$ and $C$ are obtained exactly as in \cite{OW}. For the reader's convenience, let us give the proofs. Concerning $B$, using again Taylor expansion, the boundedness of $\|f_c\|_{L^\infty}$, and
by  Cauchy-Schwarz inequality, we obtain
\[
 B\les \int_{S} |v_0-v_c| |f_c| \ud x\les \lt(\int_{S} ((z-c)^2+1)(v-v_c)^2 \ud x\int_{S} \frac{1}{(z-c)^2+1} f_c^2 \ud x\rt)^{\frac 1 2}.
\]
It is not hard to see that (see \cite[formula (2.12)]{OW})
\[
 \int_{S} ((z-c)^2+1)(v-v_c)^2 \ud x\les (|c|+1)c^2
\]
so that applying Lemma \ref{2.1} and \eqref{nonlineardissip}, we obtain
\begin{equation}\label{B}  
B\les ((|c|+1)c^2\,D)^{\frac 1 2}\,.
\end{equation} 
Finally, for $C$, by Cauchy-Schwarz inequality and  \eqref{nonlineardissip}, we get
\begin{equation*}
C\leq\left(\int_{S}|v_{0z}-v_{cz}|^2\ud x\,\int_{S}|\nabla f_c|^2\ud x\right)^{\frac{1}{2}}\les \left(\int_{S}|v_{0z}-v_{cz}|^2\ud x\, D\right)^{\frac 1 2}\,,
\end{equation*}
which combined with
\begin{equation*}
\int_{S}|v_{0z}-v_{cz}|^2\ud x\le\min\left\{ 4\int_{S}|v_{0z}|^2\ud x,|c|^2\int_{S}|v_{0zz}|^2\ud x\right\}\les\min\{c^2,1\}\les (|c|+1)c^2,
\end{equation*}
yields
\begin{equation}\label{C}
C\les ( (|c|+1)c^2\,D)^{\frac 1 2}.
\end{equation}
In conclusion,  summing \eqref{A}, \eqref{B} and \eqref{C}, we obtain \eqref{stimeutiliintro4b}.\\
\medskip

{\it Step 2:} Proof of \eqref{stimeutiliintro4c}\,.\\
We set
\begin{equation}\label{defofg}
g:= \Delta u-G'(u),
\end{equation}
so that $D=\int_{S}|\nabla g|^2\ud x$ and the Cahn-Hilliard equation \eqref{CH} can be rewritten as 
\begin{equation}\label{rewCHC}
u_t =-\Delta g\,.
\end{equation}
Moreover, set
\begin{equation}\label{defofh}
h:= \Delta g= - u_t\,.
\end{equation}
By differentiating, integrating by parts and using \eqref{rewCHC}, we have
\begin{eqnarray}\nonumber
 \frac{1}{2}\frac{\ud D}{\ud t} &=& \int_{S} \nabla (-\Delta u+G'(u)) \cdot \nabla(-\Delta u_t+G''(u)u_t) \ud x \\
\nonumber&=& \int_{S} \Delta g (-\Delta +G''(u))  u_t \ud x \\
\nonumber&=& \int_{S} \Delta g \Delta^2 g \ud x -\int_{S} |\Delta g|^2 G''(u) \ud x \\
\nonumber&=& -\int_{S} |\nabla \Delta g|^2 \ud x -\int_{S} |\Delta g|^2 G''(u) \ud x \\
\nonumber&=& -\int_{S} |\nabla h|^2 +G''(v_c) h^2\ud x - \int_{S} (G''(u)-G''(v_c)) h^2 \ud x\\ \label{dopo2.22}
&=&-I-\textit{II}\,.
\end{eqnarray}
We now decompose $h$ as 
\begin{equation}\label{decomph}
h=h_0+\alpha v_{cz}\,,
\end{equation} 
where
\begin{equation}\label{2.23}
\alpha:= \frac{\displaystyle \int_{S} h \, v_{cz} \ud x}{\displaystyle \int_{S} v^2_{cz} \ud x}\,,
\end{equation}
implying $\int_{S} h_0 v_{cz}\ud x =0$. For further use, we notice that 
\begin{equation}\label{futuro}
|\alpha| \leq \displaystyle \frac{ \displaystyle \int_{S}|\nabla g| v_{czz}\ud x}{\displaystyle \int_{S} v_{cz}^2\ud x} \leq \frac{\displaystyle \left(\int_{S}|\nabla g|^2 \ud x\right)^{\frac{1}{2}}\left(\int_{S} v^2_{czz}\ud x\right)^{\frac{1}{2}}}{\displaystyle \int_{S} v_{cz}^2\ud x} \lesssim D^{\frac{1}{2}}\,.
 \end{equation}
Moreover, by differentiating the first equation in \eqref{ELv} we obtain that $G''(v_c)v_{cz}= v_{czzz}$, whence, integrating by parts we deduce
\begin{equation}\label{forsetolgo1}
\int_{S}v^2_{czz}+G''(v_c)v^2_{cz}\ud x=\int_{\R}v^2_{czz}+G''(v_c)v^2_{cz}\ud z=0\,.
\end{equation}
Furthermore by integration by parts again,
\begin{equation}\label{forsetolgo2}
\int_{S}\nabla h_0\cdot \nabla v_{cz}+G''(v_c)\,h_0\,v_{cz}\ud x=0\,.
\end{equation}
Let us estimate $I$ in \eqref{dopo2.22}. As a consequence of \eqref{forsetolgo1} and \eqref{forsetolgo2},  by applying Lemma \ref{lingap}  to $h_0$ and using that $G''(v_c)\ge -1$, we obtain 
\begin{equation}\label{2.24}
\begin{aligned}
I
&= \int_{S} |\nabla h_0|^2 +G''(v_c) h_0^2\ud x \\
&= (1-\beta) \lt[\int_{S} |\nabla h_0|^2 +G''(v_c) h_0^2\ud x \rt]+ \beta \lt[\int_{S} |\nabla h_0|^2 +G''(v_c) h_0^2\ud x\rt]\\
&\geq C (1-\beta)\lt[\int_{S} h_0^2 \rt]+ \beta\lt[\int_{S}|\nabla h_0|^2 - h_0^2\ud x\rt] \\
&\ges \int_{S} |\nabla h_0|^2 + h_0^2\ud x, 
\end{aligned}
\end{equation}
where $\beta>0$ is chosen such that $C(1-\beta)-\beta>0$.
\vskip0.2cm
We now turn our attention to the term $II$ of \eqref{dopo2.22}. By  boundedness of $f_c$, the decomposition of $h$ in \eqref{decomph}, and Young inequality, we get 
\begin{equation}\label{somma}
\begin{aligned}
|\textit{II}|&\le \int_{S} |G''(u)-G''(v_c)| h^2 \ud x\\
  &\les \int_{S} |f_c| h^2 \ud x \\
& \les \int_{S} |f_c|\, h_0^2 \ud x + \alpha^2 \int_{S}   |f_c|\, v_{cz}^2 \ud x\, \\
&=: A+B\,.
\end{aligned}
\end{equation}
In order to bound $A$, we use \eqref{gagliahd}, which yields 
\begin{equation}\label{primastimaA}
\begin{aligned}
A \leq \|f_c\|_{L^\infty} \int_{S} h_0^2 \ud x \les \E^{\frac 1 2-\frac {d'}{12}} D^{\frac{d'}{12}}\int_{S} h_0^2 \ud x\,.
\end{aligned}
\end{equation}
Moreover, by using - in order of appearance - \eqref{decomph}, \eqref{defofh},  Cauchy-Schwarz inequality, and \eqref{futuro}, we obtain
\begin{equation*}
\begin{aligned}
\int_{S} h_0^2 \ud x &= \int_{S} h_0 (h-\alpha v_{cz}) \ud x \\
& \leq \left|\int_{S}h_0\, \Delta g \ud x\right|+ |\alpha|\left| \int_{S}h_0 v_{cz} \ud x\right| \\
& \leq \left|\int_{S}\nabla h_0\cdot \nabla g \ud x\right| + |\alpha| \left( \int_{S} h_0^2\ud x \int_{S} v_{cz}^2\ud x\right)^{\frac{1}{2}} \\
&\les \left( \int_{S} |\nabla h_0|^2 \ud x\int_{S} |\nabla g|^2\ud x\right)^{\frac{1}{2}} + |\alpha|\left( \int_{S} h_0^2\ud x\right)^{\frac{1}{2}}  \\
&\les D^{\frac{1}{2}} \left(\int_{S} |\nabla h_0|^2 + h_0^2\ud x\right)^{\frac{1}{2}},
\end{aligned}
\end{equation*}
which, together with \eqref{primastimaA}, implies
 \begin{equation}\label{stimafinA}
  A \les \E^{\frac 1 2-\frac{d'}{12}} D^{\frac 1 2+\frac{d'}{12}} \left(\int_{S} |\nabla h_0|^2 + h_0^2\ud x\right)^{\frac{1}{2}}.
  \end{equation}
As for the term $B$, by using -- in order of appearance -- Cauchy-Schwarz inequality, the exponential decay of $v_{cz}$,  Lemma \ref{2.1}, and  \eqref{nonlineardissip}, we obtain
\begin{equation}\label{stimafinB}
\begin{aligned}
B& \leq \alpha^2\left( \int_{S} f_c^2v_{cz}^2\ud x\,\int_{S} v_{cz}^2\ud x\right) ^{\frac{1}{2}} \\
& \les\alpha^2 \left( \int_{S} \frac{f_c^2}{(z-c)^2+1}\ud x\right)^{\frac{1}{2}} \\
& \les \alpha^2 \left( \int_{S} |\nabla f_c|^2 \ud x\right)^{\frac{1}{2}} \\
& \les D \, D^{\frac{1}{2}}= D^{\frac{3}{2}}\,.
\end{aligned}
\end{equation}  
Thus, by \eqref{somma}, \eqref{stimafinA}, and \eqref{stimafinB}, we end up with
  \begin{equation}\label{2.25}
|\textit{II}|\les \E^{\frac 1 2-\frac{d'}{12}} D^{\frac 1 2+\frac{d'}{12}}\left( \int_{S}|\nabla h_0|^2 + h_0^2\ud x\right)^{\frac{1}{2}}+D^{\frac{3}{2}}.
  \end{equation}
Finally, by \eqref{dopo2.22}, \eqref{2.24}, \eqref{2.25} and by Young inequality, we get for $\e$ small enough
 \begin{equation*}
 \begin{aligned}
\frac{\ud D}{\ud t}  &\les -  \int_{S} |\nabla h_0|^2 + h_0^2\ud x+ \E^{\frac 1 2-\frac{d'}{12}} D^{\frac 1 2+\frac{d'}{12}} \left(  \int_{S}|\nabla h_0|^2 + h_0^2\ud x\right)^{\frac{1}{2}}+D^{\frac{3}{2}} \\
& \les-  \int_{S} |\nabla h_0|^2 + h_0^2\ud x + \frac {1}{\ep} \E^{1-\frac{d'}{6}} D^{1+\frac{d'}{6}}+ \e  \int_{S} |\nabla h_0|^2 + h_0^2\ud x + D^{\frac{3}{2}} \\
&\les  -\int_{S} |\nabla h_0|^2 + h_0^2\ud x+D^{\frac{3}{2}} +\E^{1-\frac{d'}{6}} D^{1+\frac{d'}{6}} \\
& \les D^{\frac{3}{2}} +\E^{1-\frac {d'}{6}} D^{1+\frac{d'}{6}},
\end{aligned}
 \end{equation*}
 which concludes  the proof of \eqref{stimeutiliintro4c}.
\end{proof}
The last ingredient is an ODE argument using the relations obtained in Lemma \ref{lemmaintro} and Lemma \ref{lemmaintrobis}. This is described in the following lemma, which is the counterpart of \cite[Lemma 1.5]{OW}. 
\begin{lemma}\label{lem:ode}
 For  $3\le d\le 5$ and $c_\star\ge1$, let  $\E$, $D$, $H$ and $c$ be positive quantities related in a time interval $[0,t_\star]$ by the differential relations
 \begin{equation}\label{dif:ode}
  \frac{\ud \E}{\ud t}  =-D, \quad \frac{\ud H}{\ud t}\les c_\star^{\frac 1 2}( (c^2 D)^{\frac 1 2} +  \E^{\frac{3}{2}-\frac{d'}{12}}D^{\frac{d'}{12}}), \quad \frac{\ud D}{\ud t}\les D^{\frac 3 2}+ \E^{1-\frac{d'}{6}}D^{1+\frac{d'}{6}},
 \end{equation}
 where $d':=\max(d,3)$ and by the algebraic relations
\begin{equation}\label{alg:ode}
 \E\les (HD)^{\frac 1 2}+c_\star^2 D \qquad \textrm{and } \qquad c^2\les (H\E)^{\frac 1 2} +c_\star \E\,;
\end{equation}
Then, letting $\mathcal{G}_0:=H_0+c_\star^2(1+\E_0^2) \E_0$, it holds
\begin{eqnarray}
\label{stimasuE:ode}&\E(t) &\le \E_0 \\ 
\label{1intro:ode} &\E(t) &\les \mathcal{G}_0t^{-1}, \\ 
\label{2introbis:ode}&c^2(t) &\les  \mathcal{G}_0^{\frac{1}{2}}\E^{\frac{1}{2}}_0 \\ 
\label{2intro:ode} &c^2(t) &\les \mathcal{G}_0t^{-\frac{1}{2}} \\ 
\label{stimasuH:ode} &H(t)& \les \mathcal{G}_0 \\ 
\label{stimadissipazione:ode} &D(t) & \les \lt(\mathcal{G}_0+ \mathcal{G}_0^2 + \E_0\,\mathcal{G}_0^{\frac{6}{6-d'}}t^{-\frac{2d'-6}{6-d'}}\rt)t^{-2}.
\end{eqnarray}
\end{lemma}
We   refer the reader to \cite[Theorem 1.2]{OW} for the derivation of Proposition \ref{prop:decay} from Lemma \ref{lemmaintro}, Lemma \ref{lemmaintrobis} and Lemma \ref{lem:ode} (and in particular for better understanding the role of $c_\star$).

\begin{remark}
\rm{
 Notice that if $\E_0\les 1$, all the estimates besides \eqref{stimadissipazione:ode} coincide with those obtained in \cite[Lemma 1.5]{OW}. Still assuming $\E_0\les 1$, 
 also \eqref{stimadissipazione:ode} reduces to its counterpart in \cite{OW} provided  $t$ is larger than a suitable constant (depending only on $d'$, $H_0$, $\E_0$ and $c_\star$). 
}
\end{remark}
\begin{proof}[Proof of Lemma \ref{lem:ode}]
 Estimate \eqref{stimasuE:ode} is a direct consequence of $\frac{\ud \E}{\ud t}  =-D\le0$. Moreover, \eqref{1intro:ode}, \eqref{2introbis:ode} and \eqref{2intro:ode} follow exactly as in \cite[Lemma 1.5]{OW} replacing $c_\star^2$ by $c_\star^2(1+\E_0^2)$ in the proofs.
 
 \noindent
 Let us show \eqref{stimasuH:ode}.
 Since $\E$ is decreasing in time, we can make the change of variable $t\leftrightarrow \E$ which in light of the first equation in \eqref{dif:ode} gives $-\frac{\ud}{\ud\E}=D^{-1} \frac{\ud}{\ud t}$ so that plugging the second inequality of \eqref{alg:ode} into the second one in \eqref{dif:ode} we get
 \[
  -\frac{\ud H}{\ud\E}\les c_\star^{\frac 1 2}\lt( (H\E)^{\frac 1 4} D^{-\frac 1 2} +(c_\star \E)^{\frac 1 2} D^{-\frac 1 2} + \E^{\frac{3}{2}-\frac{d'}{12}} D^{\frac{d'}{12}-1}\rt).
 \]
Using that the first estimate in \eqref{alg:ode} implies that $D^{-1}\les H \E^{-2} +c_\star^2 \E^{-1}$, this gives
\begin{align}\nonumber
 -\frac{\ud H}{\ud\E}&\les c_\star^{\frac 1 2}\lt(H^{\frac 3 4} \E^{-\frac 3 4}+c_\star H^{\frac 1 4} \E^{-\frac 1 4}+c_\star^{\frac 1 2} H^{\frac 1 2} \E^{-\frac 1 2}+c_\star^{\frac 3 2}\rt.\\ \nonumber
 &\qquad \qquad \lt.+c_\star^{2-\frac{d'}{6}}\E^{\frac 1 2}+\E^{\frac{d'}{12}-\frac{1}{2}}H^{1-\frac{d'}{12}} \rt)\\ \label{derH}
 &\les c_\star^{\frac 1 2}\lt(H^{\frac 3 4} \E^{-\frac 3 4}+c_\star^{\frac 3 2}+ \E_0^{\frac 1 2}\lt( c_\star^{2-\frac{d'}{6}}+\lt(H^{\frac 1 4} \E^{-\frac 1 4}\rt)^{4-\frac{d'}{3}}\rt)\rt),
\end{align}
where we used Young inequality together with \eqref{stimasuE:ode}. Notice that $0< 4-\frac{d'}{3}\le 3$ and $0<2-\frac{d'}{6}\le \frac{3}{2}$ since $3\le d'\le 5$\,. Therefore  we can use again Young inequality  together with $c_\star\ge 1$ 
to get
\[
c_\star^{2-\frac{d'}{6}}+\lt(H^{\frac 1 4} \E^{-\frac 1 4}\rt)^{4-\frac{d'}{3}}\les c_\star^{\frac 3 2}+ H^{\frac 3 4} \E^{-\frac 3 4} +1\les c_\star^{\frac 3 2}+H^{\frac 3 4} \E^{-\frac 3 4},
\]
which plugged into \eqref{derH} implies 
\[
-\frac{\ud H}{\ud \E}\le C_0 c_\star^{\frac 1 2}\lt( 1+\E_0^{\frac 1 2}\rt) \lt(H^{\frac 3 4} \E^{-\frac 3 4}+c_\star^{\frac 3 2}\rt)\,
\]
for some constant $C_0>0$\,.
From this we deduce as in \cite{OW} that
\[
 -\frac{\ud}{\ud\E} \lt(H+C_0 \, c_\star^2\lt(1+\E_0^{\frac 1 2}\rt)\E\rt)^{\frac 1 4}\le C_0 \lt(1+\E_0^{\frac 1 2}\rt) \frac{\ud}{\ud\E} (c_\star^2 \E)^{\frac 1 4},
\]
which after integration gives 
\[
 H\les H_0+ (1+\E_0^{\frac 1 2})c_\star^2 \E+ c_\star^2(1+\E_0^{\frac 1 2})^4  \E_0\les H_0+ c_\star^2(1+\E_0^2) \E_0=\mathcal{G}_0
\]
and \eqref{stimasuH:ode} is proven.\\
We finally prove \eqref{stimadissipazione:ode}. To ease notation set $\gamma:= \E_0^{1-\frac{d'}{6}}$. Using \eqref{stimasuE:ode} and the fact that $\E\le \E_0$, the last estimate in \eqref{dif:ode} can be rewritten as
\begin{equation}\label{derD}
 -\frac{\ud}{\ud t} \lt( \frac{1}{\max\lt(\gamma D^{\frac {d'}{6}},D^{\frac 1 2}\rt)}\rt)\les 1.\,
 \end{equation}
\noindent 
Fix $T\in (0,t_\star)$.
Integrating \eqref{derD} between $s$ and $T$ (with $0<s<T$) we find
\begin{equation}\label{estimDmax}
 \max\lt(\gamma D^{\frac {d'}{6}}(s),D^{\frac 1 2}(s)\rt)\ges \frac{ \max\lt(\gamma D^{\frac{d'}{6}}(T),D^{\frac 1 2}(T)\rt)}{1+\max\lt( \gamma D^{\frac {d'}{6}}(T),D^{\frac 1 2}(T)\rt) (T-s) }.
\end{equation}

Let $t\in (0,T)$.
By the first equality in \eqref{dif:ode} and by  \eqref{1intro:ode}, we have
\begin{equation}\label{estimDintegral}
 \int_t^T D(s) \ud s\les \mathcal{G}_0t^{-1}.
\end{equation}
Assume first that $\gamma\ges D(T)^{-\frac{d'-3}{6}}$  and let $A:=\{s\in [t,T]\,:\,\gamma\ges D(s)^{-\frac{d'-3}{6}}\}$. Combining \eqref{estimDintegral} with \eqref{estimDmax}, we obtain
\begin{multline*}
 \mathcal{G}_0t^{-1}\ges D(T)\int_t^T \chi_A \frac{1}{\lt(1+ \gamma D^{\frac{d'}{6}}(T)(T-s)\rt)^{\frac{6}{d'}}} \ud s \\
 +\gamma^2 D^{\frac{d'}{3}}(T)\int_t^T \chi_{A^c} \frac{1}{\lt(1+\gamma  D^{\frac{d'}{6}}(T)(T-s)\rt)^{2}} \ud s
\end{multline*}
Using once again that $\gamma\ges D(T)^{-\frac{d'-3}{6}}$ so that $\gamma^2 D^{\frac{d'}{3}}(T)\ges D(T)$, and that 
 \[
 \min\lt( \frac{1}{(1+x)^{\frac{6}{d'}}},\frac{1}{(1+x)^2}\rt)\ges \chi_{[0,1]}(x)\qquad\textrm{for }x\ge 0\,,
 \]
we deduce
\begin{multline*}
 \mathcal{G}_0t^{-1}\ges D(T)\int_t^T\chi_{[0,1]}(\gamma D^{\frac{d'}{6}}(T)(T-s) )\ud s\\
 \ges \gamma^{-1} D^{1-\frac{d'}{6}}(T)\int_0^{\gamma D^{\frac{d'}{6}}(T) (T-t)} \chi_{[0,1]}(x)\ud x\ges \min(\gamma^{-1} D^{1-\frac{d'}{6}}(T), D(T) (T-t) )\,,
 \end{multline*}
 where the second inequality follows from the change of variable $x=\gamma D^{\frac{d'}{6}}(T)(T-s)$\,.
Taking $t=\frac T 2$ we get
\[
 D(T)\les \frac{\mathcal{G}_0+ \gamma^{\frac{6}{6-d'}}\mathcal{G}_0^{\frac{6}{6-d'}}T^{-\frac{2d'-6}{6-d'}}}{T^2}\,,
\]
which using that $\gamma=\E_0^{1-\frac{d'}{6}}$ gives
\begin{equation}\label{primocaso}
  D(T)\les \frac{\mathcal{G}_0+ \E_0\,\mathcal{G}_0^{\frac{6}{6-d'}}T^{-\frac{2d'-6}{6-d'}}}{T^2}\,.
\end{equation}
Arguing analogously in the case $\gamma\les D(T)^{-\frac{d'-3}{6}}$, for any $0<t<T$ one can show
$$
\mathcal{G}_0t^{-1}\ges \min(D(T)(T-t),D^{\frac{1}{2}}(T))
$$
from which we deduce
\begin{equation}\label{secondocaso}
D(T)\les \frac{\mathcal{G}_0+\mathcal{G}_0^2}{T^{2}}\,.
\end{equation}
In conclusion  \eqref{stimadissipazione:ode} follows by summing \eqref{primocaso} and \eqref{secondocaso}.
\end{proof}

  \section{Global existence}\label{section:globalex}
In this section we prove global existence and uniqueness for solutions of  the Cahn-Hilliard equation \eqref{CH}
under the assumptions of Theorem \ref{teointro0}.
We start by stating a local-in-time existence and uniqueness result, whose proof is postponed to the end of this section. 
\begin{theorem}\label{Tex0}
Let $\bdelta>0$ be given.  For any $T>0$, there exists $\delta=\delta(T)>0$ such that for every   $v_\ciniz\in\K$ and ${\iniz}\in  L^\infty(S)\cap (v_\ciniz+H^1(S))$ with $\|\iniz-v_\ciniz\|_{L^\infty}\le\delta$, 
the problem
\begin{equation}\label{CHC0}
\left\{\begin{array}{l}
u_t-\Delta (G'(u)-\Delta u)=0 
\\
 u(0)=\iniz 
\end{array}\right.
\end{equation}
admits a unique $C^\infty$ 
solution $u$  on $[0,T]$.  Moreover, 
\begin{equation}\label{tesitex1}
\| u(t)-v_{c(t)}\|_{L^\infty} \leq \dfin \quad \forall \ t \in [0,T]\,,
\end{equation}
where $v_{c(t)}$ is given by \eqref{minimale} and
  $u(t)\in v_{c(t)} +H^j(S)$ for every $j\in\N\cup\{0\}$ and  $t \in (0, T]$\,.
\end{theorem}
\begin{remark}
\rm{
We point out that our local-in-time existence result holds in any space dimension. 
}
\end{remark}


We now prove that thanks to the relaxation estimates established in Proposition \ref{prop:decay}, we can pass from a local to a global existence result for solutions of \eqref{CHC0}.

\begin{theorem}\label{globalex}
Let $2\le d\le 5$ and let $\bdelta$ be given by Proposition \ref{nonlinear}. For every $\overline{H}>0$ and $\overline{\E}>0$, there exists $\delta>0$ such that for every  $u_0$ 
with $ \E_0:=\E(u_0)\le \overline{\E}$, $H_0:= H(u_0)\le \overline{H}$ and $\|u_0-v_0\|_{L^\infty}\le\delta$,
equation \eqref{CHC0} admits a unique global smooth solution $u$ with
\begin{equation}\label{tesitexglobal}
\| u(t)-v_{c(t)}\|_{L^\infty} \leq \bdelta, \quad \forall \ t\ge 0\,.
\end{equation}
Moreover,  for every $t>0$ and for every $j\geq 0$, there holds: $u(t)\in (v_{c(t)} +  H^j(S))$\,,
\begin{equation}\label{finitezza}
\E(u(t)) < \infty, \qquad H(u(t)) <\infty.
\end{equation}
\end{theorem}

\begin{proof}
Let us start by proving that if a solution $u$ exists with $u(t)\in (v_{c(t)}+H^4(S))$, then \eqref{finitezza} holds true.
We first point out that since $d\le 5$, by Gagliardo-Nirenberg inequality (see the proof of \eqref{gagliahd}), $u(t)\in L^\infty(S)$, and hence $u(t)-v_{c(t)}\in L^\infty(S)$. 
By arguing as in \eqref{primopasso} and \eqref{stima4}, we get
\begin{align*}
 \E(u(t))&\les \|u(t)-v_{c(t)}\|^2_{H^1}<\infty\,.
\end{align*}
{As a consequence, using that $u(t)\in (v_{c(t)}+H^4(S))$ solves \eqref{CHC0}, the energy gap $\E(u(t))$ is differentiable with respect to $t$ and
\begin{equation}\label{differgap}
\E(\iniz) - \E (u(t)) = \int_{0}^t \| \nabla \left(G'(u(s))-\Delta u(s) \right) \|^2_{L^2} \ud s\,.
\end{equation}
}

We now turn to $H(u)$. Since
\begin{equation*}
H(u)=\| u(t)-v_0 \|^2_{\dot H^{-1}} \les \| u(t)-\iniz \|^2_{\dot H^{-1}} +  H(\iniz),
\end{equation*}
we only need to prove that  $\| u(t)-\iniz \|^2_{\dot H^{-1}} $ is finite.

By \eqref{differgap} and \eqref{CHC0}, using also the definition of the $\dot H^{-1}$ norm
and Jensen inequality, we get
\begin{equation*}
\begin{aligned}
\| u(t)-\iniz \|^2_{\dot H^{-1}} &= \left\| \int_0^t u_t(s) \ud s \right\|^2_{\dot H^{-1}} = \left\| \int_0^t\Delta(G'(u(s))-\Delta u(s))\ud s \right\|^2_{\dot H^{-1}}\\
&\le\int_{S}\left|\int_0^t|\nabla(G'(u(s))-\Delta u(s))|\ud s\right|^2\ud x\\
&\le t\int_{0}^t\ud s\,\int_{S}|\nabla (G'(u(s))-\Delta u(s))|^2\ud x\\
&=t(\E(\iniz)-\E(u(t))),
\end{aligned}
\end{equation*}
which immediately implies that $H(u)<\infty$.\\

We can now turn to the global existence result. 
Let $\delta_1:=\min\{\delta(1),\bar\delta\}$, where $\delta(1)$ is given by Theorem \ref{Tex0} for $T=1$.
Let moreover $T\gg 1$ be such that
\begin{equation}\label{condizione}
 \E_0^{\frac{1}{2}-\frac{d'}{12}}\lt(\G_0+\G_0^2\rt)^{\frac{d'}{12}} T^{-\frac{d'}{6}}\ll \delta_1,
\end{equation}
with $\G_0= H_0+\E_0+\E_0^7$, and let $\delta:=\delta(T)$ be given by Theorem \ref{Tex0}.

By Lemma \ref{lemH1}, using that $\E_0<\infty$ and $H_0<\infty$, we get that $u_0\in (v_0+H^1(S))$. Therefore, assuming that $\|u_0-v_0\|_{L^\infty}\le\delta$,
by Theorem \ref{Tex0}, there exists a unique smooth solution $u=u(t)$ to \eqref{CHC0}in $[0,T]$, with $u(t)\in (v_{c(t)}+H^j(S))$ for any $j\ge 0$ and for any $t\in (0,T]$, and 
\begin{equation}\label{tesitexglobalapp}
\|u(t)-v_{c(t)}\|_{L^\infty}\le \bar\delta \qquad\forall t\in[0,T]\,.
\end{equation}
It follows that  $u$ satisfies the assumptions of Proposition \ref{prop:decay} in $[0,T]$ with  $T\gg1$,  and hence \eqref{stimadissipazione} gives
\begin{equation}\label{secondastima}
D(u(T))\les (\G_0+\G_0^2)T^{-2}\,.
\end{equation}
Using Corollary \ref{cor:gaglia},  \eqref{stimasuE}, \eqref{secondastima} and \eqref{condizione} we obtain
\begin{equation}\label{delta1}
\begin{aligned}
\|u(T)-v_{c(T)}\|_{L^\infty}&\les \E^{\frac 1 2-\frac{d'}{12}}(u(T))\,D^{\frac {d'}{12}}(u(T))\\
&\les \E_0^{\frac 1 2-\frac{d'}{12}}\lt (\G_0+\G_0^2 \rt)^{\frac{d'}{12}}T^{-\frac{d'}{6}}\\
&\ll \delta_1\,.
\end{aligned}
\end{equation}
We can thus apply once again Theorem \ref{Tex0} with $u_0$ replaced by $u(T)$ and $v=v_{c(T)}$ in order to extend the solution $u$ to the interval $[0,T+1]$. Using \eqref{stimadissipazione} and arguing as above we get
\begin{equation*}
\|u(T+1)-v_{c(T+1)}\|_{L^\infty}\les \E_0^{\frac 1 2-\frac{d'}{12}}\lt (\G_0+\G_0^2 \rt)^{\frac{d'}{12}}(T+1)^{-\frac{d'}{6}}\le \E_0^{\frac 1 2-\frac{d'}{12}}\lt (\G_0+\G_0^2 \rt)^{\frac{d'}{12}}T^{-\frac{d'}{6}}\stackrel{\eqref{delta1}}{ \ll} \delta_1\,,
\end{equation*}
and we can now iterate this procedure in order to obtain a global solution.


\end{proof}


We may now turn to the proof of Theorem \ref{Tex0}. For this we will  use Duhamel formula for constructing the solution and then apply Banach fixed point Theorem. This strategy is inspired by \cite{LWZ}. 
 We first need some estimates on the fundamental solution of the parabolic bi-harmonic equation 
 \begin{equation}\label{biLap}
 \partial_t u+\Delta^2 u=0 \qquad \textrm{in } S\,.
 \end{equation}
We first fix some notation.
For any $\si>0$ we denote by $Q_\si$ the $d-1$ dimensional torus with sidelength $l$ and we set $S_\si:=Q_\si\times \R$, with the convention that $Q:=Q_1$ and $S:=S_1$. The set $\mathcal{S}(\R^d)$ denotes the Schwartz class.

\noindent
Using Fourier transform in the last variable and Fourier series in the variables in $Q$, we can write 
\begin{equation*}
u(t,x',\xd)=\int_{\R} \sum_{\xi'\in \Z^{d-1}}  c_{\xid}(t,\xi')  e^{2 i \pi \xi\cdot x} \ud \xid,
\end{equation*}
so that  \eqref{biLap} implies 
\begin{equation*}
\dot c_{\xid}= -(2\pi)^4|\xi|^4 c_{\xid},
\end{equation*} 
and thus 
\begin{equation*}
c_{\xid}(t,\xi')= c_{\xid}(0,\xi') e^{-|2\pi\xi|^4 t}.
\end{equation*}

The fundamental solution of \eqref{biLap} is thus 
\begin{equation}\label{kernel}
k(t,x):= \int_{\R} \sum_{\xi'\in \Z^{d-1}}  e^{-|2\pi\xi|^4 t} e^{2 i \pi \xi\cdot x}  \ud \xid,
\end{equation}

In order to prove $L^1$ estimates on the kernel $k$ and its derivatives, we will need the following lemma.
\begin{lemma}\label{fourlemma}
Let $u\in \mathcal{S}(\R^d)$. Then 
\[
 \int_{\R^{d}}\left|\int_{\R^d} u(\eta) e^{2 i \pi\eta\cdot y}\ud\eta\right|\ud y<\infty
\]
and for any $t>0$
\begin{equation}\label{finito}
\int_{S_{t^{-\frac 1 4}}} \lt| \int_\R t^{\frac {d-1} 4}\sum_{\eta'\in (t^{\frac 1 4}\Z)^{d-1}} u(\eta) e^{2 i \pi \eta\cdot y} \ud \etad\rt| \ud y\leq \int_{\R^{d}}\left|\int_{\R^{d}} u(\eta) e^{2 i \pi\eta\cdot y }\ud\eta\right|\ud y.
\end{equation}

\end{lemma}

\begin{proof}
Since $u\in \mathcal{S}(\R^d)$, also its Fourier transform is in the Schwartz class and thus
\[
 \int_{\R^{d}}\left|\int_{\R^{d}} u(\eta) e^{2 i \pi\eta\cdot y}\ud\eta\right|\ud y<\infty\,.
\]
For any $\etad\in\R$, we set $g_\etad(\cdot):=u(\cdot,\etad)$ which is then also in the Schwartz class.
Therefore, for any $y'\in\R^{d-1}$,
by Poisson summation formula, we have
\begin{equation}\label{poisson}
t^{\frac {d-1} 4}\sum_{\eta'\in (t^{\frac 1  4}\Z)^{d-1}}g_\etad(\eta')e^{ 2 i \pi\eta'\cdot y'}=\sum_{\zeta'\in  (t^{- \frac 1 4}\Z)^{d-1}} \check g_\etad(\zeta'+y'),
\end{equation}
where $\check g_\etad(\cdot)$ denotes the inverse Fourier transform of $ g_\etad(\cdot)$ defined by
\begin{equation}\label{antitr}
 \check g_\etad(\zeta')=\int_{\R^{d-1}}g_\etad(\eta')e^{2 i \pi\zeta'\cdot \eta'}\ud \eta'.
\end{equation}
By \eqref{poisson}, Fubini and the change of variables $\tilde y'=y'+\zeta'$, we have
\begin{align*}
\lefteqn{\int_{S_{t^{-\frac 1 4}}} \lt| \int_\R t^{\frac {d-1} 4}\sum_{\eta'\in (t^{\frac 1 4}\Z)^{d-1}} u(\eta',\etad) e^{2 i \pi \eta\cdot y} \ud \etad\rt| \ud y} \\ 
&=\int_{S_{t^{-\frac 1 4}}} \lt| \int_\R\sum_{\zeta'\in (t^{-\frac 1 4}\Z)^{d-1}} \check g_\etad(\zeta'+y')\,e^{2 i \pi\eta_d\cdot \yd}\ud \etad\rt| \ud y\\ 
&\le\int_\R\lt(\sum_{\zeta'\in (t^{-\frac 1 4}\Z)^{d-1}}\int_{Q_{t^{- \frac 1 4}}}\left|\int_{\R}\check g_\etad(\zeta'+y')\,e^{2 i \pi\etad\cdot\yd}\ud \etad\right|\ud y'\rt)\ud\yd\\ 
&=\int_\R\lt(\sum_{\zeta'\in (t^{- \frac 1 4}\Z)^{d-1}}\int_{-\zeta'+Q_{t^{-\frac 1 4}}}\left|\int_{\R}\check g_\etad(\tilde y')\,e^{2 i \pi\etad\cdot\yd}\ud \etad\right|\ud \tilde y'\rt)\ud\yd\\ 
&=\int_\R\int_{\R^{d-1}}\left|\int_\R \check g_\etad(\tilde y')e^{2 i \pi\etad\cdot y_d}\ud\etad\right|\ud \tilde y'\ud\yd\\ 
&=\int_{\R^{d}}\left|\int_{\R^{d}} u(\eta)e^{2i \pi  \eta\cdot y}\ud\eta\right|\ud  y\,,
\end{align*}
where the last equality follows directly from the  definition of $g_\etad$ and  \eqref{antitr}.
\end{proof}

We may now prove $L^1$ bounds for the kernel $k$ and its derivatives.

\begin{proposition}\label{lemmanucleo}
For every $j\in \N$, there exists $\gamma_j>0$ such that 
for any $t>0$, 
\begin{equation}\label{2nucleo} 
\|\nabla^j k(t) \|_{L^1(S)} \leq \gamma_j \, t^{-\frac{j} 4}.
\end{equation}
\end{proposition}

\begin{proof}
Since for $l= 1,\dots, d$ and $\alpha_l\in \N$, 
\begin{equation*}
\partial_{x_{l}}^{\alpha_l} k (t,x) = \int_{\R} \sum_{\xi' \in \Z^{d-1}} e^{-|2\pi\xi|^4 t} e^{2 i \pi \xi \cdot x} (2 i \pi \, \xi_{l})^{\alpha_l} \ud \xid\,
\end{equation*}
by the change of variables $\xi =t^{-\frac 1 4}  \eta $ and $x=t^{\frac 1 4} y$, we have for every $(\alpha_1,\dots,\alpha_d)$, with $\sum_l \alpha_l=j$
\begin{equation*}
\begin{aligned}
\int_{S} \left| \partial_{x_{1}}^{\alpha_1} \dots \partial_{x_{d}}^{\alpha_d} k(t,x)\right| \ud x &={(2\pi)^j} \int_{S}\left| \int_{\R} \sum_{\xi' \in \Z^{d-1}}\, e^{-|2\pi\xi|^4 t}\,  \xi_{1}^{\alpha_1} \, \dots \,  \xi_{d}^{\alpha_d} \,e^{2 i \pi  \xi\cdot x}  \ud \xid\right|\ud x \\
&={(2\pi)^j}\int_{S_{t^{-\frac 1 4}}} \left| \int_{\R} t^{\frac{d-1} {4}}\sum_{\eta' \in \left(t^{\frac {1} {4}}\Z\right)^{d-1}}e^{-|2\pi\eta|^4 }  \, t^{-\frac{\alpha_1}{4}} \eta_{1}^{\alpha_1} \, \dots \, t^{-\frac{\alpha_d}{4}}\eta_{d}^{\alpha_d}\,e^{2 i \pi \eta \cdot y} \ud \etad\right|\ud y \\
&={(2\pi)^j}t^{-\frac{j}{4}} \int_{S_{t^{-\frac 1 4}}} \left|\int_{\R} t^{\frac{d-1}{4}} \sum_{\eta' \in \left(t^{\frac {1} {4}}\Z\right)^{d-1}}  e^{-|\eta|^4 }  \eta_{1}^{\alpha_1} \, \dots \, \eta_{d}^{\alpha_d} e^{2 i \pi \eta\cdot y} \ud \etad\right|\ud y.
\end{aligned}
\end{equation*}
Estimate  \eqref{2nucleo} follows form  Lemma \ref{fourlemma} applied to 
$$
u(\eta)=e^{-|\eta|^4 }   \eta_{1}^{\alpha_1} \, \dots \,  \eta_{d}^{\alpha_d}\,.
$$
\end{proof}

We finally prove Theorem \ref{Tex0}.

\begin{proof}[Proof of Theorem \ref{Tex0}]
Let $\bdelta$ and $T>0$ be given. For $u_0\in L^\infty(S)$, set $\finiz:=\iniz-v_\ciniz$ and define  the operator 

\begin{multline*}
\T f (t,x):=\int_{S} k(t,x-y)\finiz(y) \ud y\\
 +\int_0^t \ud s \int_{S}\Delta k(t-s,x-y)( G'(f(s,y)+v_\ciniz(y))-G'(v_\ciniz(y))) \ud y\,.
\end{multline*}
We divide the proof into three steps.

{\it Step 1: Existence of a classical solution $u$ to \eqref{CHC0} satisfying \eqref{tesitex1}. }

\noindent
Let $T_0>0$ and set 
$$
\Str:=\{(t,x)\,:\,t\in[0, T_0],\,x\in S\}\,.
$$
Consider 
\begin{equation}\label{domainC}
 \cC:=\Big\{ f\in C^{1,4}(\Pi) \ : \ \sum_{j=0}^4 \| t^\frac{j}{4} \nabla^j f\|_{L^\infty(\Pi)}+ \|t\, \partial_t f\|_{L^\infty(\Pi)}\le 1\Big\}
\end{equation}
equipped with the natural weighted norm
\begin{equation}\label{normC}
 \|f\|_{\cC}:=\sum_{j=0}^4 \| t^\frac{j}{4} \nabla^j f\|_{L^\infty(\Pi)}+ \|t\, \partial_t f\|_{L^\infty(\Pi)}\,.
\end{equation}
We prove that $\T$ is a contraction in $\cC$ for $\|f^0\|_{L^\infty(S)}+T_0$ small enough. 
We start by proving that $\T$ leaves  $\cC$ invariant. If $f$ is in $\cC$, it is standard to check that $\T f \in C^{1,4}(\Pi)$ and that 
 \begin{equation}\label{eqsolvedbyT}
  \partial_t \T f +\Delta^2 \T f =\Delta (G'(f+v)-G'(v)).
 \end{equation}
We refer for instance to \cite{Evans} for a similar computation in the case of the heat equation. For $f\in\cC$,  
 using that $\|v_\ciniz\|_{L^\infty(S)}, \|f\|_{L^\infty(\Str)}\le 1$ we have
\begin{equation}\label{stimaG}
|G'(f+v_\ciniz)-G'(v_\ciniz)|\les  |f|,
\end{equation}
which, combined with \eqref{2nucleo}, implies that for any $t\in[0,T_0]$
\begin{equation}\label{pallainpalla}
\begin{aligned}
\|\T f(t)\|_{L^\infty(S)}
&\les\|k(t)\|_{L^{1}(S)} \finizinf+\int_{0}^{t}\|\Delta k(t-s)\|_{L^1(S)}\|f(s)\|_{L^\infty(S)}\\
&\les \finizinf+\int_{0}^t(t-s)^{-\frac 1 2}\|f(s)\|_{L^\infty(S)}\ud s\\
&\les\finizinf+t^{\frac 1 2}\|f\|_{L^\infty(\Str)}\\
&\les \finizinf+T_0^{\frac 1 2}\|f\|_{\cC}\,.
\end{aligned}
\end{equation}
and, analogously,
\begin{equation}\label{pinp1}
\begin{aligned}
 \|t^{\frac 1 4} \nabla \T f(t)\|_{L^\infty(S)}&\les  t^{\frac 1 4}\|\nabla k(t)\|_{L^1(S)}\|f^0\|_{L^\infty(S)}\\
 &\quad+ t^{\frac 1 4}\int_0^t \|\nabla \Delta k(t-s)\|_{L^1(S)} \| G'(f(s)+v)-G'(v)\|_{L^\infty(S)} \ud s\\
 &\les \|f^0\|_{L^\infty(S)}+ t^{\frac 1 4}\int_0^t {(t-s)^{-\frac 3 4}} \| f(s)\|_{L^\infty(S)} \ud s\\
 &\les \|f^0\|_{L^\infty(S)}+ T_0^{\frac 1 2} \|f\|_{\cC}\,.
\end{aligned}
\end{equation}
Moreover, again by \eqref{2nucleo}, we get
\begin{align*}
 \|t^{\frac 1 2} \nabla^2 \T f(t)\|_{L^\infty(S)}&\les  t^{\frac 1 2}\|\nabla^2 k(t)\|_{L^1(S)}\|f^0\|_{L^\infty(S)}\\
 &\quad+ t^{\frac 1 2}\int_0^t \|\nabla \Delta k(t-s)\|_{L^1(S)} \|\nabla (G'(f(s)+v)-G'(v))\|_{L^\infty(S)} \ud s\\
 &\les \|f^0\|_{L^\infty(S)}\\
 &\quad+ t^{\frac 1 2}\int_0^t {(t-s)^{-\frac 3 4}} \| (G''(f(s)+v)-G''(v))\nabla(f(s)+v) + G''(v) \nabla f(s)\|_{L^\infty(S)} \ud s.
\end{align*}
Since,  for $f\in \cC$,
\begin{multline}\label{contoder}
  \| (G''(f(s)+v)-G''(v))\nabla(f(s)+v) + G''(v) \nabla f(s)\|_{L^\infty(S)}\\\les \|f(s)\|_{L^\infty(S)} [s^{-\frac 1 4} (s^{\frac 1 4} \|\nabla f(s)\|_{L^\infty(S)})+1] + s^{-\frac 1 4} (s^{\frac 1 4} \|\nabla f(s)\|_{L^\infty(S)})
  \les s^{-\frac 1 4} \|f\|_{\cC}
\end{multline}
we have
\begin{multline}\label{pinp2}
  \|t^{\frac 1 2} \nabla^2 \T f(t)\|_{L^\infty(S)}\les\|f^0\|_{L^\infty(S)}+ t^{\frac 1 2} \|f\|_{\cC} \int_0^t {(t-s)^{-\frac 3 4} s^{-\frac 1 4}}{\ud s}\\
  \les \|f^0\|_{L^\infty(S)}+ t^{\frac 1 2} \|f\|_{\cC} \int_0^1{(1-s)^{-\frac 3 4} s^{-\frac 1 4}} {\ud s}
  \les \|f^0\|_{L^\infty(S)}+ t^{\frac 1 2} \|f\|_{\cC}\,,
\end{multline}
which yields as above 
\[\|t^{\frac 1 4} \nabla \T f\|_{L^\infty(\Pi)}\les \|f^0\|_{L^\infty(S)}+ T_0^{\frac 1 2} \|f\|_{\cC}\,.\]
By using again \eqref{2nucleo}, for $j=3,4$, we get
\begin{equation}\label{contoderj}
\begin{aligned}
 \|t^{\frac j 4} \nabla^j \T f(t)\|_{L^\infty(S)}&\les  t^{\frac j 4}\|\nabla^j k(t)\|_{L^1(S)}\|f^0\|_{L^\infty(S)}\\ 
 &\quad + t^{\frac j 4}\int_0^t \|\nabla \Delta k(t-s)\|_{L^1(S)} \|\nabla^{j-1} (G'(f(s)+v)-G'(v))\|_{L^\infty(S)} \ud s\\ 
 &\les \|f^0\|_{L^\infty(S)}\\
 &\quad+ t^{\frac j 4}\int_0^t {(t-s)^{-\frac 3 4}} \|\nabla^{j-1} (G'(f(s)+v)-G'(v))\|_{L^\infty(S)} \ud s\,.
\end{aligned}
\end{equation}
By expanding the derivatives of $G'(f(s)+v)$ and arguing as in \eqref{contoder}, it is straightforward to check 
\begin{equation}\label{servedopo}
\|\nabla^{j-1} (G'(f(s)+v)-G'(v))\|_{L^\infty(S)} \les  s^{-\frac {j-1}{ 4}} \|f\|_{\cC}\, 
\end{equation}
which together with \eqref{contoderj} yields
\begin{equation}\label{pinpj}
\|t^{\frac j 4} \nabla^j \T f(t)\|_{L^\infty(S)}\les \|f^0\|_{L^\infty(S)}+ T_0^{\frac 1 2} \|f\|_{\cC}.
\end{equation}
for $j=3,4$ and $t\in [0,T_0]$\,.
By summing \eqref{pallainpalla}, \eqref{pinp1}, \eqref{pinp2} and \eqref{pinpj} we get
\[
 \sum_{j=0}^4  \|t^{\frac j 4} \nabla^j \T f(t)\|_{L^\infty(S)}\les \|f^0\|_{L^\infty(S)}+ T_0^{\frac 1 2} \|f\|_{\cC}.
\]
Finally, by \eqref{eqsolvedbyT} and by \eqref{servedopo}, for $t\in [0,T_0]$ we obtain
\[
 \|t \partial_t \T f(t)\|_{L^\infty(S)} \les \| t\nabla^4 \T f(t)\|_{L^\infty(S)}+ \|t \Delta (G'(f(t)+v)-G'(v))\|_{L^\infty(S)}\les \|f^0\|_{L^\infty(S)}+ T_0^{\frac 1 2} \|f\|_{\cC}\,.  
\]
We conclude that for $\|f^0\|_{L^\infty(S)}+ T_0$ small enough,
\[\| \T f\|_{\cC}\le C \lt( \|f^0\|_{L^\infty(S)}+ T_0^{\frac 1 2} \|f\|_{\cC}\rt) \le 1.\]
The contractivity of $\T$ for $T_0$ small enough is obtained similarly. Therefore, there exists a unique fixed point $f$ in $\cC$.  Set $u(t,x):=f(t,x)-v(x)$, we immediately have that $u$ is a classical solution of \eqref{CHC0} in the interval $[0,T_0]$.

Moreover, by \eqref{pallainpalla} we have
\[
{ \max_{t\in [0,T_0]}} \|f(t)\|_{L^\infty(S)}\les  \|f^0\|_{L^\infty(S)}+ T_0^{\frac 1 2} {\max_{t\in [0,T_0]}} \|f(t)\|_{L^\infty(S)},
\]
so that, for $T_0$ small enough, 
\begin{equation}\label{sirifainl2}
{ \max_{t\in [0,T_0]}} \|f(t)\|_{L^\infty(S)}\les \frac{\|f^0\|_{L^\infty(S)}}{1-C T_0^{\frac 1 2}}\les \|f^0\|_{L^\infty(S)}\,.  
\end{equation}

By dividing the interval $[0,T]$ in intervals of length $T_0$ and iterating the procedure above, we obtain that the solution $f$ can be extended to the interval $[0,T]$ and that
\begin{equation}\label{expon}
{ \max_{t\in [0,T]}} \|f(t)\|_{L^\infty(S)}\les \|f^0\|_{L^\infty(S)} e^{c T}\,.
\end{equation}
Recalling that $u(t):=f(t)-v$, by  \eqref{expon} and Lemma \ref{lemma:fc}, we get that there exists $\delta:=\delta(T)>0$ such that if $\finizinf\le \delta$, then 
$$
\|u(t)-v_{c(t)}\|_{L^\infty(S)}\le\bdelta\qquad\textrm{for any }t\in[0,T]\,.
$$

{\it Step 2: $u(t)-v_{c(t)}\in H^j(S)$ for every $t\in (0,T]$ and for every $j\in \N\cup\{0\}$\,.}

\noindent
We preliminarily notice that, by the very definition of $v_{c(t)}$ in \eqref{minimale} and by the fact that $v_{c(t)}\in \dot{H}^j(S)$,  
in order to get the claim it is enough to show that
for every $t_0\in (0,T)$ and for every $j\in\N\cup\{0\}$ there holds
\begin{equation}\label{spacereg}
\sup_{t\in(t_0,T]} \|\nabla^j f(t)\|_{L^2(S)}\le C_j
\end{equation}
where $C_j$ is a positive constant depending on $f_0$, $t_0$, $T$, and $j$\,.

\noindent
To this  purpose, we first show,  by induction on $j$, that for every $t_0\in(0,T)$ and for every $j\in\N\cup\{0\}$,  
\begin{equation}\label{inftyder}
\sup_{t\in(t_0,T]} \|\nabla^j f(t)\|_{L^\infty(S)}\le C_j'
\end{equation}
with  $C'_j$ positive and depending on $f_0$, $t_0$, $T$, and $j$\,.
By \eqref{expon}, \eqref{inftyder} holds true for $j=0$.
Let us assume now that it holds up to $j-1$. Since for $t\in (t_0,T)$,
\begin{equation}\label{semigroup}
 f(t,x)=\int_{S} k(t,x-y) f(t_0,y) \ud y +\int_{t_0}^t\int_S \Delta k(t-s,x-y)(G'(f(s,y)+v(y))-G'(v(y))) \ud y \ud s, 
\end{equation}
using Young inequality, the inductive assumption, \eqref{2nucleo}  and \eqref{expon}, we get
\begin{align*}
 \|\nabla^j f(t)\|_{L^\infty(S)}& \les \|\nabla^j k(t)\|_{L^1(S)}\| f(t_0)\|_{L^\infty(S)}\\ 
 &\quad +\int_{t_0}^{t}\|\nabla\Delta k(t-s)\|_{L^1(S)}\|\nabla^{j-1}(G'(f(s)+v_\ciniz)-G'(v_\ciniz))\|_{L^\infty(S)} \ud s\\
&\les \gamma_j\, t^{-\frac j 4}\|f(t_0)\|_{L^\infty(S)}+ C'_{j-1}\,\int_{t_0}^t (t-s)^{-\frac 3 4} \ud s\\
&\les  \gamma_j\, t_0^{-\frac j 4}\|f^0\|_{L^\infty(S)}e^{cT}+C'_{j-1}T^{\frac 1 4},
\end{align*}
which proves \eqref{inftyder}. 

\noindent
With \eqref{inftyder} in hand, we can prove that \eqref{spacereg} holds true.
We proceed once again by induction on $j$. As for $j=0$, 
by using in order of appearance \eqref{semigroup}, 
\eqref{stimaG}, Minkowski and Young inequalities and
\eqref{2nucleo}, for any $0<t<T_0$, we obtain 
\begin{eqnarray*}
\|f(t)\|_{L^2(S)}&\les&\|k(t)\|_{L^1(S)}\|\finiz\|_{L^2(S)}+\int_{0}^t\|\Delta k(t-s)\|_{L^1(S)}\|G'(f(s)+v_\ciniz)-G'(v_\ciniz)\|_{L^2(S)} \ud s\\
&\les& \|\finiz\|_{L^2(S)}+\int_{0}^{t}(t-s)^{-\frac 1 2}\|f(s)\|_{L^2(S)} \ud s\,\\
&\les& \|\finiz\|_{L^2(S)}+ T_0^{\frac 1 2}\max_{t\in [0,T_0]}\|f(t)\|_{L^2(S)},
\end{eqnarray*}
which, by arguing as in \eqref{sirifainl2} and \eqref{expon} yields
\begin{equation}\label{ldueb}
\max_{t\in [0,T]}\|f(t)\|_{L^2(S)}\les \|\finiz\|_{L^2(S)} \, e^{cT},
\end{equation}
thus proving \eqref{spacereg} for $j=0$.
 
\noindent
Let us prove a similar bound for $\nabla f(t)$. By arguing as above and using \eqref{ldueb} we have
\begin{align*}
\|\nabla f(t)\|_{L^2(S)}&\les \|k(t)\|_{L^1(S)}\|\nabla\finiz\|_{L^2(S)}+\int_{0}^{t}\|\nabla\Delta k(t-s)\|_{L^1(S)}\|G'(f(s)+v_\ciniz)-G'(v_\ciniz)\|_{L^2(S)} \ud s\\ 
&\les\|\nabla\finiz\|_{L^2(S)}+ \int_{0}^{t}(t-s)^{-\frac 3 4} \|f(s)\|_{L^2(S)}\ud s\\ 
&
{\les}  \|\nabla\finiz\|_{L^2(S)}+\|\finiz\|_{L^2(S)} \int_{0}^{t}(t-s)^{-\frac 3 4}e^{c{s}} \ud s\\ 
&\les \|\nabla\finiz\|_{L^2(S)}+\|\finiz\|_{L^2(S)}T^{\frac{1}{4}} \, e^{c{T}},
\end{align*}
This proves \eqref{spacereg} for $j=1$. 

\noindent
Let us assume that it holds up to $j-1$. By 
 \eqref{semigroup}, \eqref{stimaG}, \eqref{2nucleo}, the inductive assumption and \eqref{ldueb}, we have
\begin{align*}
\|\nabla^jf(t)\|_{L^2(S)}&\le \|\nabla ^j k(t)\|_{L^1(S)}\|f(t_0)\|_{L^2(S)}\\
&\qquad +\int_{t_0}^{t}\|\nabla\Delta k(t-s)\|_{L^1(S)}\|\nabla^{j-1}(G'(f(s)+v_\ciniz)-G'(v_\ciniz))\|_{L^2(S)} \ud s\\
&\les \gamma_j t^{-\frac j 4} \|f(t_0)\|_{L^2(S)}+C_{j-1}\int_{0}^{t}(t-s)^{-\frac 3 4}\ud s\\
&\les \gamma_j t_0^{-\frac j 4} \|f^0\|_{L^2(S)}e^{cT}+C_{j-1}\lt( t_0^{-\frac j 4} +T^{\frac 1 4}\rt),
\end{align*}
thus concluding the proof of \eqref{spacereg}. 

{\it Step 3: $C^\infty$ regularity of the solution $u$\,.}

\noindent
The regularity of $u$ follows by a bootstrap argument. We briefly sketch it.
Let $f:=u-v$ where $u$ satisfies \eqref{CHC0}, namely $f$ satisfies
\begin{equation}\label{auxiliary}
\left\{\begin{array}{l}
f_t+\Delta^2 f=\Delta (G'(f+v)-G'(v)) 
\\
 f(0)=f^0 
\end{array}\right.
\end{equation}
and set $g:=G''(f+v)$. 
Let $t_0>0$ and $h^0\in L^2(S)\cap L^\infty(S)$. Consider the problem
\begin{equation}\label{bootpb}
\left\{\begin{array}{l}
h_t+\Delta^2 h=\Delta(g h)\qquad\textrm{in }S\\
h(t_0)=h^0\quad
\end{array}\right.
\end{equation}
and note that {\it formally} $h=f_t$ for $t\ge t_0$\,. We first show that \eqref{bootpb} admits a unique classical solution $h$. This would imply in particular that  $h\in C^{1,4}$ so that $f\in C^{2,4}((0,T)\times S)$.\\
Let $T_0>0$ and  
\begin{eqnarray*}
\Pi'&:=&\{(t,x)\,:\,t\in [t_0,T_0],\,x\in S\}\\
\cC'&:=&\Big\{ h\in C^{1,4}(\Pi) \ : \ \sum_{j=0}^4 \| t^\frac{j}{4} \nabla^j h\|_{L^\infty(\Pi')}+ \|t\, \partial_t h\|_{L^\infty(\Pi')}\le M\|h^0\|_{L^\infty(S)}\Big\}\,,
\end{eqnarray*}
with $M>0$ to be determined.
Let moreover $\|\cdot\|_{\cC'}$ be defined as in \eqref{normC} by replacing $\Pi$ with $\Pi'$.
By arguing as in Step 1, one can show there exists $M>0$ and a unique classical solution $h$ to the problem \eqref{bootpb} in the interval $[t_0,T_0]$. Moreover, by arguing as in \eqref{sirifainl2} and \eqref{expon}, one can see that the solution $h$ can be extended to the interval $[t_0,T]$.
The iteration of the argument above applied to all the derivatives in space-time of $f$ yields the desired result.
\end{proof}

\end{document}